\documentclass[11pt, notitlepage]{article}
\usepackage{amssymb,amsmath,comment}
\usepackage{amsthm}
\catcode`\@=11 \@addtoreset{equation}{section}
\def\thesection{\arabic{section}}

\def\theequation{\thesection.\arabic{equation}}
\catcode`\@=12
\usepackage{refcheck}
\usepackage{colortbl}
\usepackage{hyperref}
\usepackage[mathscr]{eucal}
\usepackage{epsf}
\usepackage{esint}
\usepackage{a4wide}

\newcommand{\Om} {\Omega}

\newcommand{\noi} {\noindent}

\usepackage[all]{xy}
\catcode`\@=11
\def\theequation{\@arabic{\c@section}.\@arabic{\c@equation}}
\catcode`\@=12

\newtheorem{Theorem}{Theorem}[section]
\newtheorem{Lemma}[Theorem]{Lemma}

\newtheorem{Corollary}[Theorem]{Corollary}
\newtheorem{Remark}[Theorem]{Remark}
\newtheorem{Definition}[Theorem]{Definition}

\begin{document}

{\vspace{0.01in}}

\title{Nonlocal singular problem and associated Sobolev type inequality with extremal in the Heisenberg group}

\author{Prashanta Garain}

\maketitle

\begin{abstract}\noindent
We study a fractional 
$p$-Laplace equation involving a variable exponent singular nonlinearity in the framework of the Heisenberg group. We first establish the existence and regularity of weak solutions. In the case of a constant singular exponent, we further prove the uniqueness of solutions and characterize the extremals of a related Sobolev-type inequality. Additionally, we demonstrate a connection between the solutions of the singular problem and these extremals. To the best of our knowledge, these findings provide new insights even in the classical case 
$p=2$.
\end{abstract}

\maketitle

\noi {Keywords: Fractional $p$-Laplace equation, variable singular exponent, existence, uniqueness, regularity, extremal, Sobolev type inequality, Heisenberg group.}

\noi{\textit{2020 Mathematics Subject Classification: 35J62, 35J75, 35R03, 35R11, 35A02, 35A23, 35B45.}

\bigskip

\tableofcontents

\section{Introduction}
In this article, we investigate the existence and regularity properties of weak solutions to the following nonlocal singular elliptic problem with a variable singular exponent:
\begin{equation}\label{meqn}
L_{s,p}u(x)={f(x)}{u(x)^{-\delta(x)}}\text{ in }\Omega,\,u>0\text{ in }\Om,\,u=0\text{ in }\mathbb{H}^N\setminus\Om.
\end{equation}
Here, $0<s<1<p<\infty$, $\Omega\subset \mathbb{H}^N$ is a bounded smooth domain, and the nonlocal operator $L_{s,p}$ is given by
$$
L_{s,p}u(x)=\text{P.V.}\int_{\mathbb{H}^n}\frac{|u(x)-u(y)|^{p-2}(u(x)-u(y))}{|y^{-1}\circ x|^{Q+sp}}\,dy.
$$
where P.V. denotes the principal value and $Q = 2N + 2$ is the homogeneous dimension of the Heisenberg group $\mathbb{H}^N$. Here, $f\in L^m(\Om)\setminus\{0\}$ is a nonnegative function for some $m\geq 1$, and $\delta:\overline{\Om}\to(0,\infty)$ is a continuous function (see the main results for precise assumptions). In the case where $\delta$ is a positive constant, we establish the uniqueness of weak solutions to problem \eqref{meqn}. Furthermore, we identify the extremal functions for the associated Sobolev-type inequality and show that these extremals are closely related to the weak solutions of equation \eqref{meqn}.

We emphasize that the nonlinearity on the right-hand side of equation \eqref{meqn} features a singular behavior, characterized by the blow-up as $u \to 0^+$. This is captured by the pointwise variable exponent $\delta(x) > 0$, which introduces additional complexity into the analysis. While singular elliptic problems have been extensively studied in the Euclidean setting, both for local and nonlocal operators, much less is known in the non-Euclidean framework, particularly on the Heisenberg group.

Singular elliptic problems have garnered considerable interest in the mathematical literature over the past few decades. In the Euclidean framework, both local and nonlocal formulations have been the subject of extensive investigation.

In the local setting, a prototypical example is the equation
\begin{equation}\label{plap}
-\Delta_p u = -\text{div}(|\nabla u|^{p-2}\nabla u) = f(x)u^{-\delta(x)} \quad \text{in } \Omega,\quad u > 0 \text{ in } \Omega,\quad u = 0 \text{ on } \partial\Omega,
\end{equation}
posed in a bounded smooth domain $\Omega \subset \mathbb{R}^N$, where 
$f$ is an integrable function. When the singular exponent 
$\delta>0$ is constant, the existence of classical solutions of \eqref{plap} has been discussed in \cite{CRT, Lz}, while the existence, uniqueness along other results of weak solutions of \eqref{plap} has been addressed in \cite{BGmm, Boc-Or, Caninouni, Canino, DeCave, Garainmn} and the references therein. For the case where 
$\delta>0$ varies with position, the existence of weak solutions for 
$p=2$ was established in \cite{CMP}, with subsequent extensions to the quasilinear case found in \cite{Alves, BGM, Garainmm} and references therein. Additionally, perturbed singular nonlinearities in the $p$-Laplace framework have been studied in \cite{BGmm, Radu, GST, Oliva, OP}, among others. For constant $\delta\in(0,1)$, the inequality 
\begin{equation*}
S\left(\int_{\Om}|u|^{1-\delta}\,dx\right)^\frac{1}{q}\leq\left(\int_{\Om}|\nabla u|^p\,dx\right)^\frac{1}{p},\,\,1<q<p^{*}:={Np}/{(N-p)},
\end{equation*}
for any function $u\in W_0^{1,p}(\Om)$, $1\leq p<N$, where $S$ is the Sobolev constant, was studied in \cite{GFA} for any $1 < p < \infty$, where they showed that
\begin{equation*}
\nu(\Om):=\inf_{u\in W_0^{1,p}(\Om)\setminus\{0\}}\left\{\int_{\Om}|\nabla u|^p\,dx:\int_{\Om}|u|^{1-\delta}\,dx=1\right\} \end{equation*}
is attained by a solution $u_\delta\in W_0^{1,p}(\Om)$ of the singular $p$-Laplace equation, \begin{equation}\label{lclssbeqn} -\Delta_p u=\nu(\Om)u^{-\delta}\text{ in }\Om,\,u>0\text{ in }\Om,\,u=0\text{ on }\partial\Om. \end{equation} We also refer to \cite{BGmm22, BGdie, EP} and the references therein for related results.

In the nonlocal context, the equation
\begin{equation*}\label{fplap}
\begin{split}
(-\Delta_p)^s u(x) &= \text{P.V.} \int_{\mathbb{R}^N} \frac{|u(x) - u(y)|^{p-2}(u(x) - u(y))}{|x - y|^{N+sp}}\,dy = f(x)u^{-\delta(x)} \,\, \text{in } \Omega,\\
u&>0\text{ in }\Om,\quad u = 0 \text{ in } \mathbb{R}^N \setminus \Omega,
\end{split}
\end{equation*}
has also been the focus of recent research, for some integrable function $f$ and $1<p<\infty$ with $0<s<1$. For constant 
$\delta>0$, we refer to \cite{Nsop, Ncp, Caninoetal, Caninon, Nmn, Mukanona1, Mukanona2, Nsong, Ncvpde}, while results for variable 
$\delta$ can be found in \cite{Garaincpaa} and references therein. Moreover, in the nonlocal context, for any constant $\delta\in(0,1)$, the Sobolev-type inequality in fractional Sobolev spaces was studied by \cite{Nmn} for $u \in W_0^{s,p}(\Omega)$ with $0 < s < 1 < p < \infty$. Specifically, for a given nonnegative function $f \in L^m(\Omega) \setminus \{0\}$ with some $m \geq 1$, they showed that
\begin{equation*}
\zeta(\Omega) := \inf_{\substack{u \in W_0^{s,p}(\Omega) \setminus \{0\}}} 
\left\{ 
\int_{\Omega} \int_{\Omega} \frac{|u(x) - u(y)|^p}{|x - y|^{N + ps}} \, dx dy : \int_{\Omega} |u|^{1-\delta} f \, dx = 1 
\right\}
\end{equation*}
is attained at a solution $u_\delta \in W_0^{s,p}(\Omega)$ of the singular fractional $p$-Laplace equation
\begin{equation}\label{lclssbeqn1}
(-\Delta_p)^s u = \zeta(\Omega) f(x) u^{-\delta} \quad \text{in } \Omega, \quad u > 0 \text{ in } \Omega, \quad u = 0 \text{ in } \mathbb{R}^N \setminus \Omega.
\end{equation}
See also \cite{EPS} for related results.

In contrast, the theory in non-Euclidean settings, such as the Heisenberg group, which is also a non-commutative group or more general stratified Lie groups, remains largely undeveloped. For the local case, some recent progress has been made, see \cite{GK, GUaamp, KD, Pucciom, Pucciacv, Sahu, WW}, but in the nonlocal case, results are sparse. A notable exception is the work in \cite{GKR}, where existence results were established for constant singular exponent $\delta \in (0,1)$, and multiplicity was shown for a perturbed nonlocal problem on a stratified Lie group.

To the best of our knowledge, apart from \cite{GKR}, the problem \eqref{meqn}, especially in the setting of variable singular exponents, has not been addressed in the non-Euclidean context, even in the linear case ($p = 2$). Our goal in this work is to fill this gap by establishing existence and regularity results for weak solutions to \eqref{meqn}, through the combination of nonlocal analysis, approximation techniques, and functional tools tailored to the Heisenberg group. Moreover, we obtain uniqueness result for any constant $\delta>0$. Finally when $\delta\in(0,1)$ is a constant, we show the connection between the solutions and extremal function of the associated Sobolev type inequality.

One of the primary challenges in analyzing singular problems like \eqref{meqn} is the lack of regularity of solutions near the singularity, which prevents the direct application of standard regularity theory. To overcome this, we adopt an approximation scheme inspired by the works \cite{Boc-Or, Caninoetal, CMP} to obtain existence and regularity results. In this method, we consider regularized problems and derive uniform a priori estimates for the approximating solutions. A key ingredient in our approach is the uniform positivity of the approximate solutions on compact subsets of $\Omega$, which ensures the well-definedness of the singular term and facilitates the passage to the limit. To this end, we use a regularity result for nonlocal equations on the Heisenberg group, recently established in \cite{Picci}.

To establish the uniqueness result, we primarily follow the method developed in \cite{Caninoetal}, which is based on proving a weak comparison principle. To derive the extremal function for the associated Sobolev-type inequality, we adopt the approximation technique introduced in \cite{Nmn}.

\subsection*{Notation and Assumptions}

We collect here the main notations and assumptions used throughout the paper:
\begin{itemize}
\item For $l > 1$, the conjugate exponent is denoted by $l' = \frac{l}{l-1}$.
\item We assume $0 < s < 1$ and $1 < p < \infty$ such that $sp < Q$, where $Q = 2N + 2$ is the homogeneous dimension of $\mathbb{H}^N$.
\item The critical fractional Sobolev exponent is denoted by $p_s^* = \frac{Qp}{Q - sp}$ for $sp<Q$.
\item For $p > 1$, we define $J_p(t) = |t|^{p-2}t$ for $t \in \mathbb{R}$.
\item We denote $d\mu = \frac{dx, dy}{|y^{-1} \circ x|^{Q + sp}}$.

\item $\Om$ will denote a bounded smooth domain in $\mathbb{H}^N$ with $N\geq 2$.

\item By $\omega\Subset\Om$, we mean $\omega$ is compactly contained in $\Om$, that is $\omega\subset\overline{\omega}\subset\Om$.

\item For $u\in HW_0^{s,p}(\Om)$, we denote the norm of $u$ by $\|u\|$ as defined in \eqref{norm}, that is
$$
\|u\|:=[u]_{HW_0^{s,p}(\Om)}=\left(\int_{\mathbb{H}^N}\int_{\mathbb{H}^N}\frac{|{u(x)-u(y)}|^p}{|{\eta^{-1}\circ \xi|}^{Q+sp}}\,dx dy\right)^\frac{1}{p}.
$$
Further, if $u\in L^p(\Om)$ for some $1\leq p\leq \infty$, then the $L^p$ norm of $u$ will be denoted by $\|u\|_{L^p(\Om)}$.

\item For $k\in\mathbb{R}$, we denote by $k^+=\max\{k,0\}$, $k^-=\max\{-k,0\}$ and $k_-=\min\{k,0\}$.

\item For a function $F$ defined over a set $S$ and some constants $c$ and $d$, by $c\leq F\leq d$ in $S$, we mean $c\leq F\leq d$ almost everywhere in $S$. 

\item We use $C$ to denote a generic positive constant, which may vary from one occurrence to another, even within the same line. When the constant depends on specific parameters $r_1,r_2,\cdots,r_k$, we indicate this by writing $C=C(r_1,r_2,\cdots,r_k)$.
\end{itemize}
Before presenting our main results, we introduce the condition $(P_{\epsilon,\delta_*})$. To the best of our knowledge, the results established in this work are new even in the case of a constant singular exponent $\delta$ and $p=2$.\\
\textbf{Condition $(P_{\epsilon,\delta_*})$:} A continuous function $\delta : \overline{\Omega} \to (0,\infty)$ is said to satisfy the condition $(P_{\epsilon,\delta_*})$ if there exist constants $\delta_* \geq 1$ and $\epsilon > 0$ such that $\delta(x) \leq \delta^*$ for every $x \in \Omega_\epsilon$, where
\[
\Omega_\epsilon := \{ y \in \Omega : \mathrm{dist}(y, \partial \Omega) < \epsilon \}.
\]
Further, we define
\begin{equation}\label{deltanbd2}
\omega_\epsilon=\Omega\setminus\overline{\Omega_\epsilon},\quad\epsilon>0.
\end{equation}

\subsection{Existence results} 
\begin{Theorem}\label{varthm1}(Variable singular exponent)
 Assume that $\delta:\overline{\Om}\to(0,\infty)$ is a continuous function satisfying the condition $(P_{\epsilon,\delta_*})$ for some $\epsilon>0$ and for some $\delta_*\geq 1$. Then the problem \eqref{meqn} admits a weak solution $u\in HW_{\mathrm{loc}}^{s,p}(\Om)\cap L^{p-1}(\Om)$ such that $u^\frac{\delta_{*}+p-1}{p}\in HW_0^{s,p}(\Omega)$, provided $f\in L^m(\Omega)\setminus\{0\}$ is nonnegative, where $m=
\Big(\frac{(\delta_{*}+p-1)p_s^{*}}{p\delta_{*}}\Big)^{'}$.
\end{Theorem}
Our next theorem provide existence results in the constant singular exponent case.
\begin{Theorem}\label{thm1}(Constant singular exponent)
Assume that $\delta:\overline{\Om}\to(0,\infty)$ is a constant function.
Assume that $f\in L^m(\Omega)\setminus\{0\}$ is nonnegative for some $m$.
Then the problem \eqref{meqn} admits a weak solution $u$ in each of the following cases:
\begin{enumerate}
    \item[(a)] $0<\delta<1,\,m=\big(\frac{p_{s}^{*}}{1-\delta}\big)'$ with $u\in HW_0^{s,p}(\Omega)$.
    \item[(b)] $\delta=1,\,m=1$ with $u\in HW_0^{s,p}(\Omega)$.
    \item[(c)] $\delta>1,\,m=1$ with $u\in HW_{\mathrm{loc}}^{s,p}(\Om)\cap L^{p-1}(\Om)$ such that $u^\frac{\delta+p-1}{p}\in HW_0^{s,p}(\Omega)$.  
\end{enumerate}
\end{Theorem}

\begin{Remark}\label{exrmk}
If both $\delta $ and $\delta_*$ are strictly greater than $1$, then we observe that the solutions obtained in Theorem \ref{varthm1} and \ref{thm1}-$(c)$ also belong to the spaces $HW_0^{\frac{sp}{\delta_*+p-1},\delta_*+p-1}(\Om)$ and $HW_0^{\frac{sp}{\delta+p-1},\delta+p-1}(\Om)$ respectively. Indeed, let $u$ be such solution, then $u^\frac{\delta_*+p-1}{p}\in HW_0^{s,p}(\Omega)$ by Theorem \ref{varthm1}. Since $\frac{\delta_*+p-1}{p}>1$, using the H\"olderianity of the mapping $t\to t^\frac{\delta_*+p-1}{p}$, we have
$$
\frac{|u(x)-u(y)|^p}{|y^{-1}\circ x|^{Q+sp}}\leq \frac{\Big|u\frac{\delta_*+p-1}{p}(x)-u\frac{\delta_*+p-1}{p}(y)\Big|^p}{|y^{-1}\circ x|^{Q+sp}}
$$
for every $x,y\in\mathbb{H}^N$, which proves $u\in HW_0^{\frac{sp}{\delta_*+p-1},\delta_*+p-1}(\Om)$. The proof is same for $\delta>1$. 
\end{Remark}

\subsection{Regularity results}
\begin{Theorem}\label{regthm}(Variable singular exponent)
Suppose $\delta:\overline{\Omega}\to(0,\infty)$ is a continuous function satisfying $(P_{\epsilon,\delta_*})$ for some $\delta_*\geq 1$ and for some $\epsilon>0$. Assume that $u$ is the weak solution of the problem \eqref{meqn} given by Theorem \ref{varthm1} and $f\in L^m(\Om)\setminus\{0\}$ be nonnegative for some $m$. Then 
\begin{enumerate}
    \item[(i)] if $m\in\big[\frac{Q(\delta_*+p-1)}{Q(p-1)+\delta_* sp},\frac{Q}{sp}\big)$, then $u\in L^t(\Om)$, where $t=m'\gamma$, with $\gamma=\frac{N(p-1)(m-1)}{m(N-sp)-N(m-1)}$.
    \item[(ii)] if $m>\frac{Q}{sp}$, then $u\in L^\infty(\Om)$.
\end{enumerate}
\end{Theorem}

\begin{Theorem}\label{regthm1}(Constant singular exponent)
Let $\delta:\overline{\Omega}\to(0,1)$ be a constant function. Suppose $u$ is the weak solution of the problem \eqref{meqn} given by Theorem \ref{thm1}-$(a)$ and $f\in L^m(\Om)\setminus\{0\}$ is nonnegative for some $m$. Then 
\begin{enumerate}
    \item[(i)] if $m\in\big(\big(\frac{p_s^{*}}{1-\delta}\big)',\frac{Q}{sp}\big)$, then $u\in L^t(\Om)$, where $t=p_s^{*}\gamma$, with $\gamma=\frac{(\delta+p-1)\gamma'}{pm'-p_s^{*}}$.
    \item[(ii)] if $m>\frac{Q}{sp}$, then $u\in L^\infty(\Om)$.
\end{enumerate}
\end{Theorem}

\begin{Theorem}\label{regthm2}(Constant singular exponent)
Let $\delta\equiv 1$ in $\overline{\Om}$. Suppose $u$ is the weak solution of the problem \eqref{meqn} given by Theorem \ref{thm1}-$(b)$ and $f\in L^m(\Om)\setminus\{0\}$ is nonnegative for some $m$. Then 
\begin{enumerate}
    \item[(i)] if $m\in\big(1,\frac{Q}{sp}\big)$, then $u\in L^t(\Om)$, where $t=p_s^{*}\gamma$, with $\gamma=\frac{pm'}{pm'-p_s^{*}}$.
    \item[(ii)] if $m>\frac{Q}{sp}$, then $u\in L^\infty(\Om)$.
\end{enumerate}
\end{Theorem}

\begin{Theorem}\label{regthm3}(Constant singular exponent)
Let $\delta:\overline{\Omega}\to(1,\infty)$ be a constant function. Suppose $u$ is the weak solution of the problem \eqref{meqn} given by Theorem \ref{thm1}-$(c)$ and $f\in L^m(\Om)\setminus\{0\}$ is nonnegative for some $m$. Then 
\begin{enumerate}
    \item[(i)] if $m\in\big(1,\frac{Q}{sp}\big)$, then $u\in L^t(\Om)$, where $t=p_s^{*}\gamma$, with $\gamma=\frac{(\delta+p-1)m'}{pm'-p_s^{*}}$.
    \item[(ii)] if $m>\frac{Q}{sp}$, then $u\in L^\infty(\Om)$.
\end{enumerate}
\end{Theorem}

\subsection{Uniqueness result}
\begin{Theorem}\label{thm4}(Uniquness)
Let $\delta:\overline{\Om}\to (0,\infty)$ be a constant function. Assume that $f\in L^1(\Om)\setminus\{0\}$ be nonnegative. Then the problem \eqref{meqn} admits at most one weak solution in $HW^{s,p}_{\mathrm{loc}}(\Omega)\cap L^{p-1}(\Omega)$.
\end{Theorem}

\subsection{Sobolev type inequality with extremal}
\begin{Theorem}\label{thm5}
Let $\delta:\overline{\Om}\to (0,1)$ be a constant function. Suppose that $f \in L^m(\Omega) \setminus \{0\}$ is nonnegative, where $m=\Big(\frac{p_s^{*}}{1-\delta}\Big)'$,  and let $u_\delta \in HW_0^{s,p}(\Omega)$ be the weak solution of the problem \eqref{meqn} given by Theorem \ref{thm1}-$(a)$. Then we have
\begin{enumerate}
\item[(a)] \textbf{(Extremal)}  
\begin{multline*}
\Theta(\Omega) := \inf_{v \in HW_0^{s,p}(\Omega) \setminus \{0\}} \left\{ \int_{\mathbb{H}^N} \int_{\mathbb{H}^N} \frac{|v(x) - v(y)|^p}{|y^{-1} \circ x|^{Q + sp}} \, dx \, dy : \int_\Omega |v|^{1-\delta} f \, dx = 1 \right\} \\
= \left( \int_{\mathbb{H}^N} \int_{\mathbb{H}^N} \frac{|u_\delta(x) - u_\delta(y)|^p}{|y^{-1} \circ x|^{Q + sp}} \, dx \, dy \right)^{\frac{1-\delta-p}{1-\delta}}.
\end{multline*}

\item[(b)] \textbf{(Sobolev type inequality)} Moreover, for every $v \in HW_0^{s,p}(\Omega)$, the following mixed Sobolev inequality holds:
\begin{equation}\label{inequality2}
C \left( \int_\Omega |v|^{1-\delta} f \, dx \right)^{\frac{p}{1-\delta}} \leq \int_{\mathbb{R}^N} \int_{\mathbb{R}^N} \frac{|v(x) - v(y)|^p}{|y^{-1} \circ x|^{Q + sp}} \, dx \, dy,
\end{equation}
if and only if 
$$
C \leq \Theta(\Omega).
$$

\item[(c)] \textbf{(Simplicity)} If for some $w\in HW_0^{s,p}(\Om)$, the equality
\begin{equation}\label{sim}
\Theta(\Omega) \left( \int_\Omega |w|^{1-\delta}f \, dx \right)^{\frac{p}{1-\delta}} = \int_{\mathbb{H}^N} \int_{\mathbb{H}^N} \frac{|w(x) - w(y)|^p}{|y^{-1} \circ x|^{Q + sp}} \, dx \, dy,
\end{equation}
holds, then $w = k u_\delta$ for some constant $k$.
\end{enumerate}
\end{Theorem}

\begin{Corollary}\label{ansiowgtrmk}
From Theorem \ref{thm5}, we have
\begin{equation*}
\Theta(\Omega) = \int_{\mathbb{H}^N} \int_{\mathbb{H}^N} \frac{|V_\delta(x) - V_\delta(y)|^p}{|y^{-1} \circ x|^{Q + sp}} \, dx \, dy.
\end{equation*}
Moreover, $V_\delta \in S_\delta$ and satisfies the following singular problem
\begin{equation*}
\mathcal{L}_{s,p} V_\delta = \Theta(\Omega) f V_\delta^{-\delta} \quad \text{in } \Omega, \quad V_\delta > 0 \quad \text{in } \Omega.
\end{equation*}
\end{Corollary}
\textbf{Organization of the paper:} The remainder of this article is structured as follows:
In Section 2, we introduce the functional framework and recall some auxiliary results that will be used throughout the paper.
In Section 3, we establish some preliminary results required to prove our main results.
Finally, in Section 4, we complete the proof of our main results.

\section{Functional setting and auxiliary results}
\subsection*{The Heisenberg group}
We begin by recalling some definitions and preliminary results concerning the Heisenberg group. For a more detailed exposition, we refer the reader to the monograph~\cite{BLU07}.\\
The Heisenberg group $\mathbb{H}^N$ is a Lie group whose underlying manifold is ~$\mathbb{R}^N_x\times \mathbb{R}^N_y \times \mathbb{R}_t$. The group law~$\circ$ and dilations~$ \delta_\lambda$ are given by:
     $$
		\xi \circ \xi' := \big(x+x',\, y+y',\, t+t'+2\langle y,x'\rangle-2\langle x,y'\rangle \big), \qquad \forall \xi=(x,y,t), \xi'= (x',y',t') \in \mathbb{H}^N
	$$
     and 
     $$
         \delta_\lambda(\xi):=\big(\lambda x, \, \lambda y, \,\lambda^2 t\big),
 	 $$
	 respectively.
     It is straightforward to verify that the identity element of $\mathbb{H}^N$
 is the origin $0$, and the inverse of an element $\xi$ is given by $\xi^{-1} = - \xi$. The Jacobian determinant of the dilation $\delta_\lambda$ is equal to $\lambda^{2N+2}$, and the quantity $Q:=2N+2$ is referred to as the homogeneous dimension of $\mathbb{H}^N$.

 The canonical (or Jacobian) basis of the Heisenberg Lie algebra $\mathfrak{h}^N$
 associated with $\mathbb{H}^N$ is given by the following vector fields:
	
	The Jacobian basis of the Heisenberg Lie algebra~$\mathfrak{h}^N$ of~$\mathbb{H}^N$ is given by the following vector fields
	$$
	X_j := \partial_{x_j} +2y_j \partial_t, \quad
	X_{N+j}:= \partial_{y_j}-2x_j \partial_t, \quad 1 \leq j \leq N, \quad
	T := \partial_t.
	$$
    Moreover, it can be easily verified that
	$$
	[X_j,X_{N+j}]:= X_j X_{N+j}-X_{N+j}X_j = -4\partial_t, \quad \text{for every}~1 \leq j \leq N,
	$$
	so that
	$$
    \textup{rank}\Big(\textup{Lie}\{X_1,\dots,X_{2N}\}\Big)\ = \ 2N+1.
	$$
	The group $\mathbb{H}^N$ is thus a {\it Carnot group} with the following stratification~$\mathfrak{h}^N$ of the Lie algebra,
	$$
	\mathfrak{h}^n = \textup{span}\{X_1,\dots,X_{2N}\}\oplus \textup{span}\{T\}.
	$$
    Given a domain~$\Omega \subset \mathbb{H}^N$, for any~$u\in C^1(\Omega;\,\mathbb{R})$ we define the {\it horizontal gradient}~$\nabla_{\mathbb{H}^N} u$  of~$u$ by
	$$
	\nabla_{\mathbb{H}^N} u (\xi):= \Big(X_1u(\xi),\dots, X_{2N}u(\xi)\Big).
    $$    
	\begin{Definition} 
    A homogeneous norm on $\mathbb{H}^N$ is a continuous function (with respect to the Euclidean topology) ${d_{\rm o}} : \mathbb{H}^N \rightarrow [0,+\infty)$ satisfying the following properties:
		\begin{enumerate}
			\item[(i)]
				${d_{\rm o}}(\delta_\lambda(\xi))=\lambda {d_{\rm o}}(\xi)$, for every~$\lambda>0$ and every~$\xi \in \mathbb{H}^N$;
			\item[(ii)]
				${d_{\rm o}}(\xi)=0$ if and only if~$\xi=0$.
		\end{enumerate}
		Moreover, we say that the homogeneous norm~${d_{\rm o}}$ is {\rm symmetric} if~$
		{d_{\rm o}}(\xi^{-1})={d_{\rm o}}(\xi)$, for any~$\xi \in \mathbb{H}^N$.
	\end{Definition}
    Given a fixed homogeneous norm $d_{\rm o}$ on~$\mathbb{H}^N$, define the function $\Psi$ on the set of all pairs of elements in 
$\mathbb{H}^N$ by:
		$$
		\Psi(\xi,\eta):={d_{\rm o}} (\eta^{-1}\circ \xi),
		$$
		is a pseudometric on $\mathbb{H}^N$.
		
		\vspace{2mm}
        Homogeneous norms are generally not proper norms on $\mathbb{H}^N$. Nevertheless, for any homogeneous norm $d_{\rm o}$ on~$\mathbb{H}^N$, there exists a constant~$\Lambda>0$ such that:
		$$
	   \Lambda^{-1}	{|\xi|}_{\mathbb{H}^N} \le d_{\rm o}(\xi) \le \Lambda {|\xi|}_{\mathbb{H}^N}, \quad \forall \xi \in \mathbb{H}^N,~{|\xi|}_{\mathbb{H}^N}:= \left((|{x}|^2 +|{y}|^2)^2+t^2\right)^\frac{1}{4}.
		$$
        The function~$|{\cdot}|_{\mathbb{H}^N}$ defined above is known as {\it the Kor\'anyi distance} on the Heisenberg group, and it indeed defined a norm. For the proof, we refer to~\cite{Cyg81} and Example~5.1 in~\cite{BFS17}.
        
       	For any fixed~$\xi_0 \in \mathbb{H}^N$ and~$R>0$, we denote by~$B_R(\xi_0)$ the ball with center~$\xi_0$ and radius~$R$, defined as follows:
     	$$
     	B_R(\xi_0):=\Big\{\xi \in \mathbb{H}^N : |\xi_0^{-1}\circ \xi|_{\mathbb{H}^N} < R\Big\}.
	    $$
	\subsection*{The fractional Sobolev spaces}
    We now recall several definitions and fundamental results related to the fractional functional framework used in this work. For more comprehensive details, the reader is referred to \cite{AM,KS18}; see also \cite{DPV12} for an introduction to fractional Sobolev spaces in the Euclidean setting.
    
	Let~$p \in (1,\infty)$,~$s \in (0,1)$, and let~$u : \mathbb{H}^N \rightarrow \mathbb{R}$ be a measurable function. 
	The fractional Sobolev spaces~$HW^{s,p}$ on the Heisenberg group is defined by
    $$
		HW^{s,p}(\mathbb{H}^N):=\Big\{u \in L^p(\mathbb{H}^N):  \frac{|{u(x)-u(y)}|}{|{y^{-1}\circ x|}^{s+\frac{Q}{p}}} \in L^p(\mathbb{H}^N \times \mathbb{H}^N)\Big\},
    $$
	endowed with the natural fractional norm
    $$
			\|u\|_{HW^{s,p}(\mathbb{H}^N)}:= \Big(\int_{\mathbb{H}^N}|{u}|^p \, {\rm d}\xi +\underbrace{\int_{\mathbb{H}^N}\int_{\mathbb{H}^N}\frac{|{u(x)-u(y)}|^p}{|{y^{-1}\circ x|}^{Q+sp}}\, {\rm d}\xi {\rm d}\eta}_{=:[u]_{HW^{s,p}(\mathbb{H}^N)}}\Big)^\frac{1}{p}, \qquad u \in HW^{s,p}(\mathbb{H}^N).
    $$
     In a similar fashion, given a bounded Lipschitz domain~$\Omega \subset \mathbb{H}^N$, one can define the  fractional Sobolev space~$HW^{s,p}(\Omega)$. Further, as in \cite{GKR}, we define the space $HW^{s,p}_0(\Omega)$ as  
     $$
     HW^{s,p}_0(\Omega)=\{u\in HW^{s,p}(\mathbb{H}^N):u=0\text{ on }\mathbb{H}^N\setminus\Om\}.
     $$
     By \cite[Lemma 1]{GKR}, the space $HW^{s,p}_0(\Omega)$ is a reflexive Banach space. Further, from \cite[Theorem 1.1]{AM} (see also \cite[Theorem 2.5]{KS18}), we have
     $$
     \|u\|_{L^p(\Om)}\leq C[u]_{HW^{s,p}(\mathbb{H}^N)}
     $$
     for every $u\in C_c^{\infty}(\Om)$ with some constant $C>0$. Hence, we define an equivalent norm $\|\cdot\|$ on the space $HW_0^{s,p}(\Om)$ by
     \begin{equation}\label{norm}
     \|u\|:=[u]_{HW_0^{s,p}(\Om)}=\left(\int_{\mathbb{H}^N}\int_{\mathbb{H}^N}\frac{|{u(x)-u(y)}|^p}{|{y^{-1}\circ x|}^{Q+sp}}\,dx dy\right)^\frac{1}{p}.
     \end{equation}
     We say that~$u \in HW^{s,p}_{\mathrm{loc}}(\Om)$ if $u\in HW^{s,p}(\omega)$ for every $\omega\Subset\Om$.

To provide a rigorous framework for our analysis, we provide below the notion of weak solutions. For this, first we define the the associated zero boundary condition as follows:

\begin{Definition}[Dirichlet Boundary Condition]\label{rediri}
Let $u$ be such that $u = 0$ in $\mathbb{H}^N \setminus \Omega$. We say that $u \leq 0$ on $\partial\Omega$ if, for every $\theta > 0$, it holds that $(u-\theta)^+\in HW_0^{s,p}(\Om)$. Further, we say that $u = 0$ on $\partial\Omega$ if $u$ is nonnegative and satisfies $u \leq 0$ on $\partial\Omega$.
\end{Definition}

We now define weak solutions to the problem \eqref{meqn}.

\begin{Definition}[Weak Solution]\label{wksoldef}
Let $f \in L^1(\Omega)\setminus{0}$ be nonnegative and let $\Omega\subset \mathbb{H}^N$ be a bounded Lipschitz domain. A function $u \in HW^{s,p}_{\mathrm{loc}}(\Omega) \cap L^{p-1}(\Omega)$ is called a weak solution to problem \eqref{meqn} if $u=0$ in $\mathbb{H}^N\setminus\Om$ such that $u=0$ on $\partial\Om$ as in Definition \ref{rediri} and the following conditions hold:
\begin{itemize}
\item For every subset $\omega \Subset \Omega$, there exists a constant $C = C(\omega) > 0$ such that $u \geq C$ in $\omega$;
\item For every $\phi \in C_c^1(\Omega)$, one has
\begin{equation}\label{wksoleqn}
\int_{\mathbb{H}^N}\int_{\mathbb{H}^N}\frac{|u(x)-u(y)|^{p-2}(u(x)-u(y))(\phi(x)-\phi(y))}{|y^{-1}\circ x|^{Q+sp}}\,dx dy = \int_{\Omega} f(x) u^{-\delta(x)} \phi(x)\,dx.
\end{equation}
\end{itemize}
\end{Definition}

\subsection*{Auxiliary results}
In this subsection, we present several auxiliary results. The first, taken from \cite[Theorem 9.14]{var}, plays a key role in establishing the existence of approximate solutions.
\begin{Theorem}\label{MB}
Let $V$ be a real separable reflexive Banach space and $V^*$ be the dual of $V$. Suppose that $T:V\to V^{*}$ is a coercive and demicontinuous monotone operator. Then $T$ is surjective, i.e., given any $f\in V^{*}$, there exists $u\in V$ such that $T(u)=f$. If $T$ is strictly monotone, then $T$ is also injective.  
\end{Theorem}

The following result holds along with the lines of the proof of \cite[Proposition 1.5]{Caninoetal}.
\begin{Lemma}\label{mainappos}
Let $\gamma>0$ and let $u$ be nonnegative with $u^{\max\big\{\frac{\gamma+p-1}{p},1\big\}}\in HW^{s,p}_0(\Omega).$
Then $u$ fulfills zero  Dirichlet boundary conditions in the sense of Definition \ref{rediri}.
\end{Lemma}

For the following embedding result, we refer to \cite[Theorem 1]{emb}.
\begin{Lemma}\label{emb}
Let $0<s<1$ and $1<p<\infty$ such that $sp<Q$. Then the embedding
$$
HW^{s,p}(\Om)\hookrightarrow L^r(\Om)
$$
is continuous for every $r\in[1,p_s^*]$ and compact for every $r\in[1,p_s^*)$.
\end{Lemma}
The following result follows from \cite[Lemma 3.5]{Caninoetal}.
\begin{Lemma}\label{lem:dino}
	Let $q > 1$ and $\epsilon > 0$. In the plane $\mathbb{R}^2$, using the notation $p = (x, y)$, we define:
	\[
	S_{\epsilon}^x\,:=\{x\geq \epsilon \}\cap\{y\geq 0\}\qquad  S_{\epsilon}^y\,:=\{y\geq \epsilon \}\cap\{x\geq 0\}\,.
	\]
	It follows that
	\begin{equation*}
	|x^q-y^q|\geq {\epsilon}^{q-1}|x-y|\qquad\text{in}\quad S_{\epsilon}^x\cup S_{\epsilon}^y\,.
	\end{equation*}
\end{Lemma}
For the following algebraic inequality, see \cite[Lemma 2.1]{Dama}.
\begin{Lemma}\label{alg}
Let $1<p<\infty$. Then for every $x,y\in\mathbb{R}^k$, there exists a constant $C=C(p)>0$ such that
$$
\langle J_p(x)-J_p(y),x-y\rangle\geq C(|x|+|y|)^{p-2}|x-y|^2.
$$
\end{Lemma}
Further, for the following algebraic inequality, we refer to \cite[Lemma A.2]{BP}.
\begin{Lemma}\label{BPalg}
Let $1<p<\infty$ and $g:\mathbb{R}\to\mathbb{R}$ be an increasing function. We define
$$
G(t)=\int_{0}^{t}g'(s)^\frac{1}{p}\,ds,\quad t\in\mathbb{R}.
$$
Then for every $a,b\in\mathbb{R}$, we have
$$
J_p(a-b)(g(a)-g(b))\geq |G(a)-G(b)|^p.
$$
\end{Lemma}

\section{Preliminaries}
\subsection{Approximate problem}
For $n\in\mathbb{N}$ and a nonnegative function $f\in L^1(\Om)\setminus\{0\}$, let $f_n(x):=\min\{f(x),n\}$ and consider the following approximated problem:
\begin{equation}\label{approxeqn}
L_{s,p}u ={f_n(x)}{\Big(u^{+}+\frac{1}{n}\Big)^{-\delta(x)}}\text{ in }\Omega,\quad u=0\text{ in }\mathbb{H}^N\setminus\Om.
\end{equation}
Our main aim in this section is to establish existence of $u_n$ and prove some a-priori estimates.
\begin{Lemma}\label{approx}
Let $\delta:\overline{\Om}\to(0,\infty)$ be a continuous function. For each $n\in\mathbb{N}$, there is a unique positive solution $u_n\in HW_{0}^{s,p}(\Omega)\cap L^{\infty}(\Omega)$ to \eqref{meqn}. Additionally, $u_{n+1}\geq u_n$ in $\Omega$, for every $n\in\mathbb{N}$. Moreover, for every $n\in\mathbb{N}$ and every $\omega\Subset\Omega$, there exists a constant $C(\omega)>0$ (independent of $n$) such that $u_n\geq C(\omega)>0$ in $\omega$.
\end{Lemma}
\begin{proof}
We denote the space $HW_0^{s,p}(\Om)$ by $V$ and let $V^*$ be the dual of $HW_0^{s,p}(\Om)$. We define the mapping $S:V\to V^*$ by
$$
\langle S(v),\phi \rangle=\int_{\mathbb{H}^N}\int_{\mathbb{H}^N}J_p(v(x)-v(y))(\phi(x)-\phi(y))\,d\mu-\int_{\Om}f_n(x)\Big(v^+ +\frac{1}{n}\Big)^{-\delta(x)}\phi\,dx,
$$
for every $v,\phi\in V$. It is easy to observe that $S$ is well defined, using Lemma \ref{emb} and H\"older's inequality.
\begin{itemize}
    \item Coercivity: Using Lemma \ref{emb} and H\"older's inequality, we obtain
    $$
    \langle S(v), v\rangle=\|v\|^p-\int_{\Om}f_n(x)\Big(v^+ +\frac{1}{n}\Big)^{-\delta(x)}v\,dx\geq \|v\|^p-C\|v\|,
    $$
    for some constant $C>0$. Since $1<p<\infty$, we obtain that $S$ is coercive.

    \item Demicontinuity: Let $v_k,v\in HW_0^{s,p}(\Om)$ be such that $\|v_k-v\|\to 0$ as $k\to\infty$. Therefore, up to a subsequence, still denoted by $v_k$, it follows that $v_k\to v$ a.e. in $\Om$ and the sequence
\begin{equation}\label{bdd1}
\left\{\frac{J_p(v_k(x)-v_k(y))}{|y^{-1}\circ x|^{\frac{Q+s\,p}{p'}}}\right\}_{k\in {\mathbb N}}\quad \text{is bounded in $L^{p'}({\mathbb H}^{2N})$}.
\end{equation}
By the pointwise convergence of $v_k$ to $v$, we have
$$
\frac{J_p(v_k(x)-v_k(y))}{|x - y|^{\frac{Q+s\,p}{p'}}}\to \frac{J_p(v(x)-v(y))}{|y^{-1}\circ x|^{\frac{Q+s\,p}{p'}}}
\quad \text{a.e.\ in ${\mathbb H}^{2N}$}.
$$
Using the above fact along with \eqref{bdd1}, we obtain up to a subsequence that
$$
\frac{J_p(v_k(x)-v_k(y))}{|y^{-1}\circ x|^{\frac{Q+s\,p}{p'}}}
\rightharpoonup \frac{J_p(v(x)-v(y))}{|y^{-1}\circ x|^{\frac{Q+s\,p}{p'}}}
\quad \text{weakly in $L^{p'}({\mathbb H}^{2N})$}.
$$
Then, since
$$
\frac{\phi(x) - \phi(y)}{|y^{-1}\circ x|^{\frac{Q+s\,p}{p}}}\in L^p({\mathbb H}^{2N}),
$$
we conclude that
\begin{equation}\label{demi1}
\lim_{k\to\infty}\int_{\mathbb{H}^n}\int_{\mathbb{H}^N}J_p(v_k(x)-v_k(y))(\phi(x)-\phi(y))\,d\mu=\int_{\mathbb{H}^N}\int_{\mathbb{H}^N}J_p(v(x)-v(y))(\phi(x)-\phi(y))\,d\mu,
\end{equation}
for every $\phi\in HW_0^{s,p}(\Om)$. Moreover, by the Lebesgue's dominated convergence theorem, we observe that
    \begin{equation}\label{dctap1}
    \lim_{k\to\infty}\int_{\Om}f_n(x)\Big(v_{k}^+ +\frac{1}{n}\Big)^{-\delta(x)}\phi\,dx=\int_{\Om}f_n(x)\Big(v^+ +\frac{1}{n}\Big)^{-\delta(x)}\phi\,dx,\quad \forall\phi\in HW_0^{s,p}(\Om).
    \end{equation}
Hence, combining \eqref{demi1} and \eqref{dctap1}, we arrive at 
$$
\lim_{k\to\infty}\langle S(v_k), \phi\rangle=\langle S(v), \phi\rangle,\quad \forall \phi\in HW_0^{s,p}(\Om),
$$
and hence $S$ is demicontinuous.

\item Monotonicity of $S$: Let $u_1,u_2\in V$. Then, we observe that
\begin{align*}
&\langle S(u_1)-S(u_2),u_1-u_2\rangle\\
&=\int_{\mathbb{H}^n}\int_{\mathbb{H}^n}\{J_p(u_1(x)-u_1(y))-J_p(u_2(x)-u_2(y))\}((u_1-u_2)(x)-(u_1-u_2)(y))\,d\mu\\
&\qquad-\int_{\Om}f_n(x)\Big\{\Big(u_1^{+}+\frac{1}{n}\big)^{-\delta(x)}-\Big(u_2^{+}+\frac{1}{n}\Big)^{-\delta(x)}\Big\}(u_1-u_2)\,dx.
\end{align*}
By Lemma \ref{alg}, the first integral above is nonnegative. Moreover, it is easy to check the second integral above is nonpositive. Therefore, $S$ is monotone.
\end{itemize}
Hence, by Theorem \ref{MB}, it follows that $S$ is surjective. Therefore, for every $n\in\mathbb{N}$, there exists $u_n\in V$ such that for every $\phi\in V$, we have
\begin{equation}\label{auxeqn}
\int_{\mathbb{H}^N}\int_{\mathbb{H}^N}J_p(u_n(x)-u_n(y))(\phi(x)-\phi(y))\,d\mu=\int_{\Om}f_n(x)\Big(u_{n}^+ +\frac{1}{n}\Big)^{-\delta(x)}\phi\,dx.
\end{equation}
\textbf{Nonnegativity:} Choosing $\phi=(u_n)_-:=\min\{u_n,0\}$ as a test function in \eqref{auxeqn}, we have
\begin{equation}\label{posi}
\begin{split}
\int\limits_{\mathbb{H}^N}\int\limits_{\mathbb{H}^N}J_p(u_n(x)-u_n(y))\big((u_n)_-(x)-(u_n)_-(y)\big)\,d\mu&=\int_{\Omega}f_n(x)\Big(u_n^++\frac{1}{n}\Big)^{-\delta(x)}(u_n)_-\,dx\leq 0.
\end{split}
\end{equation}
For any $x,y\in\mathbb{H}^N$, it follows from the estimate (3.13) of \cite[page 12]{GUna} that
\begin{equation}\label{pos}
J_p(u_n(x)-u_n(y))\big((u_n)_-(x)-(u_n)_-(y)\big)\geq |(u_n)_-(x)-(u_n)_-(y)|^p.
\end{equation}
Using \eqref{pos} in \eqref{posi} we obtain
$
\|(u_n)_-\|^p=0,
$
which gives $(u_n)_-=0$. Therefore, 
\begin{equation}\label{non-neg}
u_n\geq 0\text{ in } \mathbb{H}^N\text{ for every } n\in\mathbb{N}.
\end{equation}
\textbf{Monotonicity and Uniqueness:} Let $n\in\mathbb{N}$ and let $u_n,u_{n+1}\in HW_0^{s,p}(\Om)$ be two solutions of the problem \eqref{approxeqn}. Then by \eqref{non-neg}, both $u_n$ and $u_{n+1}$ are nonnegative in $\mathbb{H}^N$. Then for every $\phi\in HW_0^{s,p}(\Om)$, the following holds:
\begin{equation}\label{auxeqn11}
\begin{split}
\int\limits_{\mathbb{H}^N}\int\limits_{\mathbb{H}^N}J_p(u_n(x)-u_n(y))\big(\phi(x)-\phi(y)\big)\,d\mu=\int_{\Omega}f_n(x)\Big(u_n+\frac{1}{n}\Big)^{-\delta(x)}\,\phi\, dx
\end{split}
\end{equation}
and
\begin{equation}\label{auxeqn21}
\begin{split}
\int\limits_{\mathbb{H}^N}\int\limits_{\mathbb{H}^N}J_p(u_{n+1}(x)-u_{n+1}(y))\big(\phi(x)-\phi(y)\big)\,d\mu=\int_{\Omega}f_{n+1}(x)\Big(u_{n+1}+\frac{1}{n+1}\Big)^{-\delta(x)}\,\phi\, dx.
\end{split}
\end{equation}
Choosing $\phi=w=(u_n-u_{n+1})^+\in HW^{s,p}_0(\Omega)$ in the equations \eqref{auxeqn11} and \eqref{auxeqn21} above and subtracting the second from the first, concerning the r.h.s. (and recalling that $f_n\leq f_{n+1}$ a.e.) we get
\begin{align*}
&\int_{\Omega}{f_n(x)}{\Big(u_n+\frac{1}{n}\Big)^{-\delta(x)}} \,(u_n-u_{n+1})^+\, dx-\int_{\Omega}{f_{n+1}(x)}{\Big(u_{n+1}+\frac{1}{n+1}\Big)^{-\delta(x)}} (u_n-u_{n+1})^+\,dx\\
&\leq \int_{\Omega} f_{n+1}(u_n-u_{n+1})^+
\frac{\Big(u_{n+1}+\frac{1}{n+1}\Big)^{\delta(x)}-\Big(u_n+\frac{1}{n}\Big)^{\delta(x)}}{\Big(u_n+\frac{1}{n}\Big)^{\delta(x)}\Big(u_{n+1}+\frac{1}{n+1}\Big)^{\delta(x)}}dx\leq 0.
\end{align*}
Then, we conclude that
\begin{equation}
\label{in-neg}
\int_{\mathbb{H}^{N}}\int_{\mathbb{H}^{N}}{\big (J_p(u_n(x) - u_n(y))-J_p(u_{n+1}(x) - u_{n+1}(y))\big)\,
	(w(x) - w(y))}\,d\mu\leq 0.
\end{equation}	
Now, arguing exactly as in the proof of \cite[Lemma 9]{LL}, we get
$$
\big (J_p(u_n(x) - u_n(y))-J_p(u_{n+1}(x) - u_{n+1}(y))\big)\, (w(x) - w(y))\geq 0,\quad\text{for a.e. $(x,y)\in {\mathbb H}^{2N}$},
$$
with the  {\em strict} inequality, unless it holds
\begin{equation}
\label{concl}
(u_n(x)-u_{n+1}(x))^+=(u_n(y)-u_{n+1}(y))^+,\quad\text{for a.e. $(x,y)\in {\mathbb H}^{2N}$}.
\end{equation}
On the other hand, by \eqref{in-neg}, we have
$$
\big (I_p(u_n(x) - u_n(y))-I_p(u_{n+1}(x) - u_{n+1}(y))\big)\, (w(x) - w(y))=0,\quad\text{for a.e. $(x,y)\in {\mathbb H}^{2N}$}.
$$
Therefore, \eqref{concl} holds true, namely
\begin{equation*}
(u_n(x)-u_{n+1}(x))^+=C,\quad\text{for a.e. $(x,y)\in {\mathbb H}^{2N}$},
\end{equation*}
for some constant $C$. Since $u_n=u_{n+1}=0$ on ${\mathbb H}^{N}\setminus\Omega$ it follows that $C=0$, which implies in turn that
$u_n(x)\leq u_{n+1}(x)$, for a.e.\ $x\in\Omega$. This concludes the proof of the monotonicity of $u_n$. Uniqueness follows similarly.

\noindent
\textbf{Boundedness:}
To establish boundedness, let
$$
A(k):=\{x\in\Omega:u_n(x)\geq k\}\,\,\text{for any}\,\, k\geq 1.
$$
Taking $\phi_k:=(u_n-k)^+=\max\{u_n-k,0\}$ as a test function in \eqref{approxeqn}, and using the estimate (3.9) from \cite[page 11]{GUna}, we obtain
\begin{equation}\label{tstbd1}
\begin{split}
\|\phi_k\|^p\leq \int\limits_{\mathbb{H}^N}\int\limits_{\mathbb{H}^N}J_p(u_n(x)-u_n(y))\big(\phi_k(x)-\phi_k(y)\big)\,d\mu=\int_{\Omega}f_n(x)\Big(u_n+\frac{1}{n}\Big)^{-\delta(x)}\phi_k\,dx.
\end{split}
\end{equation}
Therefore, by Lemma \ref{emb}, using the continuity of the embedding $HW_0^{s,p}(\Om)\hookrightarrow L^l(\Omega)$ for some $l>p$, we arrive at
\begin{multline}
\label{tstbd2}
\|\phi_k\|^p\leq\int_{\Omega}n^{\delta+1}\phi_k\,dx
\leq \|n^{\delta(x)+1}\|_{L^\infty(\Om)}\int_{A(k)}(u_n-k)\,dx\leq C|A(k)|^\frac{l-1}{l}\|\phi_k\|,
\end{multline}
for some constant $C>0$ depending on $n$. Hence, we have
\begin{equation}\label{new}
    \|\phi_k\|^p \leq C|A(k)|^\frac{p(l-1)}{l(p-1)},
\end{equation}
for some positive constant $C$ depending on $n$. Now choose $h$ such that $1\leq k<h$. Then $u(x)-k \geq (h-k)$ on $A(h)$ and $A(h)\subset A(k)$. Noting this fact along with \eqref{new}, we get
\begin{multline*}
(h-k)^p|A(h)|^\frac{p}{l}\leq\left(\int_{A(h)}\big(u_n(x)-k\big)^l\,dx\right)^\frac{p}{l}\leq\left(\int_{A(k)}\big(u_n(x)-k\big)^l~dx\right)^\frac{p}{l}\leq C\|\phi_k\|^p \leq C|A(k)|^\frac{p(l-1)}{l(p-1)},
\end{multline*}
for some positive constant $C>0$ depending on $n$. Therefore, we have
$$
|A(h)|\leq\frac{C}{(h-k)^l}|A(k)|^\frac{l-1}{p-1}.
$$
Note that $\frac{l-1}{p-1}>1$. Therefore, by applying \cite[Lemma B.1]{Stam}, we get
$$
\|u_n\|_{L^\infty(\Omega)}\leq C,
$$
for some positive constant $C$ depending on $n$. Hence, $u_n\in L^\infty(\Om)$.\\
\textbf{Uniform positivity:} By \eqref{non-neg}, we have $u_1\geq 0$ in $\mathbb{H}^N$. Since $u_1=0$ in $\mathbb{H}^N\setminus\Om$ and $f\not\equiv 0$, we have $u_1\not\equiv 0$ in $\Omega$. Thus, using \cite[Theorem 3.7]{Picci} for every $\omega\Subset\Omega$, there exists a constant $C(\omega)>0$ such that $u_1\geq C(\omega)>0$ in $\Omega$. Again, by the monotonicity we have $u_n\geq u_1$ in $\Omega$ for every $n\in\mathbb{N}$. Hence, for every $\omega\Subset\Omega$ and every $n\in\mathbb{N}$,
$$
u_n(x)\geq C(\omega)>0,\text{ for }x\in\omega,
$$
where $C(\omega)>0$ is a constant independent of $n$.
\end{proof}
\begin{Remark}\label{rmkapprox}
By Lemma \ref{approx}, the monotonicity of the sequence ${u_n}$ allows us to define its pointwise limit in $\Omega$, denoted by $u$. Consequently, we have $u \geq u_n$ in $\mathbb{H}^N$ for all $n \in \mathbb{N}$. In what follows, we show that $u$ is the desired solution to problem \eqref{meqn}.
\end{Remark}

\begin{Remark}\label{altrmk}
We note that the same proof of Lemma \ref{approx} also applies in the Euclidean setting. Moreover, we emphasize that Lemma \ref{approx} can alternatively be established using Schauder's fixed point theorem, as demonstrated in \cite{Caninoetal}. For this purpose, the result stated in Lemma \ref{approxnew}—which addresses the more general non-singular case—will be instrumental. Its proof closely follows the argument used in Lemma \ref{approx}, with the primary difference being the consideration of the mapping 
$S:V\to V^*$ defined by
$$
\langle S(v),\phi \rangle=\int_{\mathbb{H}^N}\int_{\mathbb{H}^N}J_p(v(x)-v(y))(\phi(x)-\phi(y))\,d\mu-\int_{\Om}g\phi\,dx,
$$
for every $v,\phi\in V$, where $V=HW_0^{s,p}(\Om)$ and $V^*$ being the dual of $V$.
\end{Remark}
\begin{Lemma}\label{auxresult}
Let $g\in L^{\infty}(\Om)\setminus\{0\}$ be nonnegative function in $\Omega$. Then, there exists a unique solution $u\in HW_0^{s,p}(\Omega)\cap L^{\infty}(\Om)$ of the problem
\begin{equation}\label{approxnew}
\begin{split}
L_{s,p}u=g\text{ in }\Om,\quad u>0\text{ in }\Om,\quad u=0\text{ in }\mathbb{H}^N\setminus\Om.
\end{split}
\end{equation}
Furthermore, for every $\omega\Subset\Om$, there exists a constant $C(\omega)$ such that $u\geq C(\omega)>0$ in $\omega$.
\end{Lemma}

\subsection{A-priori estimates of the approximate solutions}
We now establish a series of boundedness estimates (Lemmas \ref{apun1}–\ref{apun2}) for the sequence of positive solutions $\{u_n\}_{n \in \mathbb{N}}$ to problem \eqref{approxeqn}, as provided by Lemma \ref{approx}. These estimates play a crucial role in deriving the existence and regularity results.
\begin{Lemma}\label{apun1}(Variable singular exponent)
Let $\delta:\overline{\Om}\to(0,\infty)$ be a continuous function satisfying the condition $(P_{\epsilon,\delta_*})$ for some $\epsilon>0$ and for some $\delta_*>0$. Suppose that $f\in L^m(\Omega)\setminus\{0\}$ is nonnegative, where $m=
\Big(\frac{(\delta_{*}+p-1)p_s^{*}}{p\delta_{*}}\Big)^{'}$. Assume that $\{u_n\}_{n\in\mathbb{N}}$ is the sequence, consisting of solutions of \eqref{approx} given in Lemma \ref{approx}.
\begin{enumerate}
\item[$(i)$] If $\delta_*=1$, then the sequence $\{u_n\}_{n\in\mathbb{N}}$ is uniformly bounded in $HW_0^{s,p}(\Omega)$.
\item[$(ii)$] If $\delta_*>1$, then the sequence $\left\{u_n^{\frac{\delta_*+p-1}{p}}\right\}_{n\in\mathbb{N}}$ is uniformly bounded in $HW_0^{s,p}(\Omega)$.
\end{enumerate}
\end{Lemma}

\begin{proof}
\begin{enumerate}
\item[$(i)$] Taking $\phi=u_n$ as a test function in the weak formulation of \eqref{approxeqn}, we get
\begin{equation}\label{api}
\begin{split}
\|u_n\|^p&\leq\int_{\Omega}f_n\Big(u_n+\frac{1}{n}\Big)^{-\delta(x)}u_n\,dx\leq\int_{\overline{\Omega_\epsilon}\cap\{0<v_n\leq 1\}}fu_n^{1-\delta(x)}\,dx\\&+\int_{\overline{\Omega_\epsilon}\cap\{u_n>1\}}fu_n^{1-\delta(x)}\,dx
+\int_{\omega_\epsilon}f\|C(\omega_\epsilon)^{-\delta(x)}\|_{L^\infty(\Omega)}u_n\,dx\\
&\qquad\leq\|f\|_{L^1(\Omega)}+\Big(1+\|C(\omega_\epsilon)^{-\delta(x)}\|_{L^\infty(\Omega)}\Big)\int_{\Omega}fu_n\,dx,
\end{split}
\end{equation}
where we have used the fact that $0<\delta(x)\leq 1$ for every $x\in\Omega_\epsilon$ and the property $u_n\geq C(\omega_\epsilon)>0$ in $\omega_\epsilon$ from Lemma \ref{approx}. Since $f\in L^m(\Omega)$ for $m=(p_s^{*})'$, using H\"older's inequality and Lemma \ref{emb} in \eqref{api}, we arrive at
\begin{equation*}
\begin{split}
\|u_n\|^p&\leq \|f\|_{L^1(\Omega)}+\big(1+\|C(\omega_\epsilon)^{-\delta(x)}\|_{L^\infty(\Omega)}\big)\|f\|_{L^m(\Omega)}\|u_n\|_{L^{p_s^{*}}(\Omega)}\\
&\qquad\leq \|f\|_{L^1(\Omega)}+C\|f\|_{L^m(\Omega)}\|u_n\|,
\end{split}
\end{equation*}
for some positive constant $C$, independent of $n$. Therefore, the sequence $\{u_n\}_{n\in\mathbb{N}}$ is uniformly bounded in $HW_0^{s,p}(\Omega)$.
\item[$(ii)$] Choosing $u_n^{\delta_*}$ as a test function in the weak formulation of \eqref{approx}, we obtain
\begin{equation}\label{apigrt1}
\begin{split}
&\Big\|u_n^\frac{\delta_{*}+p-1}{p}\Big\|^{p}\leq c\int_{\Omega}f_n\Big(u_n+\frac{1}{n}\Big)^{-\delta(x)}u_n^{\delta_*}\,dx\\
&\qquad\leq c\int_{\overline{\Omega_\epsilon}\cap\{0<u_n\leq 1\}}fu_n^{\delta_{*}-\delta(x)}\,dx+c\int_{\overline{\Omega_\epsilon}\cap\{u_n>1\}}fu_n^{\delta_{*}-\delta(x)}\,dx+
c\int_{\omega_\epsilon}f\|C(\omega_\epsilon)^{-\delta(x)}\|_{L^\infty(\Omega)}u_n^{\delta_*}\,dx\\
&\qquad\leq c\|f\|_{L^1(\Omega)}+c\big(1+\|C(\omega_\epsilon)^{-\delta(x)}\|_{L^\infty(\Omega)}\big)\int_{\Omega}fu_n^{\delta_*}\,dx\\
&\qquad\leq c\|f\|_{L^1(\Omega)}+c
\|f\|_{L^m(\Omega)}\left\|u_n^{\frac{\delta_{*}+p-1}{p}}\right\|^{\frac{p\delta_*}{\delta_*+p-1}},
\end{split}
\end{equation}
for some positive constant $C$, independent of $n$, where we have used Lemma \ref{emb}, the hypothesis $\|\delta\|_{L^\infty(\Omega_\epsilon)}\leq\delta_*$ along with $u_n\geq C(\omega)>0$ from Lemma \ref{approx}. Therefore, $\left\{u_{n}^\frac{\delta_{*}+p-1}{p}\right\}_{n\in\mathbb{N}}$ is uniformly bounded in $HW_0^{s,p}(\Omega)$.
\end{enumerate}
\end{proof}

\begin{Lemma}\label{apun2}(Constant singular exponent)
 Assume that $\delta:\overline{\Om}\to(0,\infty)$ is a constant function. Assume that $\{u_n\}_{n\in\mathbb{N}}$ is the sequence, consisting of solutions of \eqref{approx} given in Lemma \ref{approx}. Let $f\in L^m(\Om)\setminus\{0\}$ be nonnegative for some $m$. Then
 \begin{enumerate}
     \item[$(i)$] the sequence $\{u_n\}_{n\in\mathbb{N}}$ is uniformly bounded in $HW_0^{s,p}(\Om)$ if $0<\delta<1$ and $m=\big(\frac{p_s^{*}}{1-\delta}\big)'$.
     \item[$(ii)$] the sequence $\{u_n\}_{n\in\mathbb{N}}$ is uniformly bounded in $HW_0^{s,p}(\Om)$ if $\delta=1$ and $m=1$.
     \item[$(iii)$] the sequence $\Big\{u_n^\frac{\delta+p-1}{p}\Big\}_{n\in\mathbb{N}}$ is uniformly bounded in $HW_0^{s,p}(\Om)$ if $\delta>1$ and $m=1$.  
 \end{enumerate}
\end{Lemma}

\begin{proof}
\begin{enumerate}
\item[$(i)$] Taking $u_n$ as a test function in the weak formulation of \eqref{approxeqn}, we obtain
$$
\|u_n\|^p\leq \int_{\Om}\frac{f_n(x)u_n}{\big(u_n+\frac{1}{n}\big)^\delta}\,dx.
$$
Since $f \in L^m(\Omega)$ and $(1 - \delta)m' = p_s^*$, applying Lemma \ref{emb}, we obtain
\begin{align*}
\|u_n\|^p&\leq \int_{\Om}f u_n^{1-\delta}\,dx\leq \|f\|_{L^m(\Om)}\left(\int_{\Om}u_n^{(1-\delta)m'}\right)^\frac{1}{m'}\\
&=\|f\|_{L^m(\Om)}\left(\int_{\Om}u_n^{p_s^*}\,dx\right)^\frac{1-\delta}{p^*}\\
&\leq C\|f\|_{L^m(\Om)}\|u_n\|^{1-\delta},
\end{align*}
for some constant $C>0$, independent of $n$. Therefore, we have
$$
\|u_n\|\leq C,
$$
for some constant $C>0$, independent of $n$.
Hence, the sequence $\{u_n\}_{n\in\mathbb{N}}$ is uniformly bounded in $HW_0^{s,p}(\Om)$.

\item[$(ii)$] By taking $\phi = u_n$ as a test function in the weak formulation of \eqref{approxeqn}, we obtain
\begin{equation}\label{}
\begin{split}
\|u_n\|^p&\leq\int_{\Om}\frac{f_n u_n}{u_n+\frac{1}{n}}\,dx\leq\|f\|_{L^1(\Om)}.
\end{split}
\end{equation}
Hence, the sequence $\{u_n\}_{n \in \mathbb{N}}$ is uniformly bounded in $HW_0^{s,p}(\Omega)$.

\item[$(iii)$] Choosing $\phi=u_n^\delta$ as a test function (which belongs to $HW_0^{s,p}(\Omega)$, since $\Phi(s)=s^\delta$, $s\geq 0$ is Lipschitz on bounded intervals), we obtain
		\begin{equation}
		\label{bound}
			\Big\|u_n^\frac{\delta+p-1}{p}\Big\|^{p}
		\leq C\int_{\Omega}{f_n(x)}{\Big(u_n+\frac{1}{n}\Big)^{-\delta}}u_n^\delta\, dx\leq c\|f\|_{L^1(\Om)},
		\end{equation}
		for some constant $C>0$ independent of $n$. It follows from \eqref{bound} that the sequence $\Big\{u_n^\frac{\delta+p-1}{p}\Big\}_{n \in \mathbb{N}}$ is uniformly bounded in $HW_0^{s,p}(\Omega)$.
\end{enumerate}
\end{proof}

\subsection{Preliminaries for the uniqueness result}
In the following two subsections 4.3-4.4, we assume that $\delta>0$ is a constant function in $\overline{\Om}$ unless otherwise mentioned. We begin by defining the real-valued function $g_k$ as
\begin{equation}\nonumber
g_{k}(s)\,:=\,
\begin{cases}
\min\{s^{-\beta}\,,\,k\}\qquad\text{if $s>0$},\\
k\qquad\qquad\quad\qquad\text{if $s\leq 0$}.
\end{cases}
\end{equation}
Next, we consider the real-valued function $\Phi_k$ defined as the primitive of $g_k$ satisfying $\Phi_k(1) = 0$.

With these definitions, we introduce the functional
$J_k\,:HW^{s,p}_0(\Omega)\rightarrow [-\infty\,,\,+\infty]$ defined by
\begin{equation}\nonumber
J_k(\phi)\,:=\,\frac{1}{p}\int_{\mathbb{H}^{N}}\int_{\mathbb{H}^{N}}{|\phi(x)-\phi(y)|^p}\,d\mu-\int_{\mathbb{H}^{N}} f(x)\Phi_{k}(\phi)\,dx,\qquad\phi\in HW^{s,p}_0(\Omega)\,,
\end{equation}
where $f\in L^1(\Om)\setminus\{0\}$ is nonnegative.
Recall that for a function $z\in HW^{s,p}_{{\rm loc}}(\Omega)\cap L^{p-1}(\Omega)$ with $z\geq 0$, we say $z$ is a weak supersolution (subsolution) to \eqref{meqn} if,
if
\begin{equation*}
\int_{\mathbb{H}^{N}}\int_{\mathbb{H}^{N}}{|z(x) - z(y)|^{p-2}\, (z(x) - z(y))\, (\phi(x) - \phi(y))}\,d\mu
\underset{(\leq)}{\geq}\int_{\Omega}{f(x)}{z^{-\delta}}\,\phi\, dx\qquad \forall \phi\in C^1_c(\Omega)\,,\, \phi\geq 0\,.
\end{equation*}
Given a fixed supersolution $v$, we define $w$ as the minimizer of $J_k$ over the convex set
\[
\mathcal K\,:=\,\{\phi\in HW^{s,p}_0(\Omega)\,:\, 0\leq \phi\leq v\,\,\text{a.e. in}\,\,\Omega\}\,.
\]
By direct computation, we deduce that
\begin{equation}\label{eqminxx}
\begin{split}
&\int_{\mathbb{H}^{N}}\int_{\mathbb{H}^{N}}{|w(x) - w(y)|^{p-2}\, (w(x) - w(y))\, (\psi(x)-w(x)-(\psi(y)-w(y)))}\,d\mu\\
&\geq \int_\Omega\,f(x) \Phi_{k}'(w)(\psi-w)\quad\,\,\, \text{for}\,\,
 \psi\in w+\left(HW^{s,p}_0(\Omega)\cap L^\infty_c(\Omega)\right)\,\,\text{and}\,\,0\leq \psi\leq v\,,
 \end{split}
\end{equation}
where $L^\infty_c(\Omega)$ denotes the space of $L^\infty$ functions with compact support in $\Om$.
With this notation, the following results hold by arguments similar to those in \cite[Lemma 4.1 and Theorem 4.2]{Caninoetal}. For convenience of the reader, we present the proof below.
\begin{Lemma}\label{lemmause}
For all $\phi \in C_c^1(\Omega)$ with $\phi \geq 0$, the function $w$ satisfies
\begin{equation}\label{ksjfskgfdfjigf}
\int_{\mathbb{H}^N} \int_{\mathbb{H}^N} |w(x) - w(y)|^{p-2} (w(x) - w(y)) (\phi(x) - \phi(y)) \, d\mu
\geq \int_\Omega f(x) \, \Phi_k'(w) \, \phi \, dx.
\end{equation}
\end{Lemma}

\begin{proof}
Let us consider a real-valued function $g \in C_c^\infty(\mathbb{R})$ with $0 \leq g(t) \leq 1$, $g(t) = 1$ for $t \in [-1,1]$ and $g(t) = 0$ for $t \in (-\infty,-2] \cup [2,\infty)$. Then, for any nonnegative $\phi \in C_c^1(\Omega)$, we set
\[
\phi_h := g\left(\frac{w}{h}\right)\phi \quad \text{and} \quad \phi_{h,t} := \min\{w + t\phi_h, \, v\}
\]
with $h \geq 1$ and $t > 0$.  
We have that $\phi_{h,t} \in w + (HW_0^{s,p}(\Omega) \cap L_c^\infty(\Omega))$ and $0 \leq \phi_{h,t} \leq v$, so that, by \eqref{eqminxx}, we deduce
\[
\begin{split}
&\int_{\mathbb{H}^N} \int_{\mathbb{H}^N} |w(x) - w(y)|^{p-2} (w(x) - w(y)) (\phi_{h,t}(x) - w(x) - (\phi_{h,t}(y) - w(y))) \, d\mu \\
&\geq \int_\Omega f(x) \Phi_k'(w) (\phi_{h,t} - w) \, dx.
\end{split}
\]

By standard manipulations and again using \eqref{eqminxx} along with Lemma \ref{alg}, we obtain
\[
\begin{split}
\mathbb{I}_1 := c \int_{\mathbb{H}^N} \int_{\mathbb{H}^N} &(|\phi_{h,t}(x) - \phi_{h,t}(y)| + |w(x) - w(y)|)^{p-2} \\
&\cdot (\phi_{h,t}(x) - w(x) - (\phi_{h,t}(y) - w(y)))^2 \, d\mu \\
&\leq \int_{\mathbb{H}^N} \int_{\mathbb{H}^N} {|\phi_{h,t}(x) - \phi_{h,t}(y)|^{p-2} (\phi_{h,t}(x) - \phi_{h,t}(y)) (\phi_{h,t}(x) - w(x) - (\phi_{h,t}(y) - w(y)))}\,d\mu \\
&\quad - \int_{\mathbb{H}^N} \int_{\mathbb{H}^N} {|w(x) - w(y)|^{p-2} (w(x) - w(y)) (\phi_{h,t}(x) - w(x) - (\phi_{h,t}(y) - w(y)))}\,d\mu\\
&\leq \int_{\mathbb{H}^N} \int_{\mathbb{H}^N} |\phi_{h,t}(x) - \phi_{h,t}(y)|^{p-2} (\phi_{h,t}(x) - \phi_{h,t}(y)) (\phi_{h,t}(x) - w(x) - (\phi_{h,t}(y) - w(y))) \, d\mu \\
&\quad - \int_\Omega f(x) \Phi_k'(w) (\phi_{h,t} - w) \, dx.
\end{split}
\]
We rewrite this as
\begin{equation}\label{cicciofriccio}
\begin{split}
\mathbb{I}_1 &- \int_\Omega f(x)(\Phi_k'(\phi_{h,t}) - \Phi_k'(w)) (\phi_{h,t} - w) \, dx \\
&\leq \int_{\mathbb{H}^N} \int_{\mathbb{H}^N} |\phi_{h,t}(x) - \phi_{h,t}(y)|^{p-2} (\phi_{h,t}(x) - \phi_{h,t}(y)) (\phi_{h,t}(x) - w(x) - (\phi_{h,t}(y) - w(y))) \, d\mu \\
&\quad - \int_\Omega f(x) \Phi_k'(\phi_{h,t}) (\phi_{h,t} - w) \, dx \\
&= \int_{\mathbb{H}^N} \int_{\mathbb{H}^N} \mathcal{G}(x,y) \, dx \, dy - \int_\Omega f(x) \Phi_k'(\phi_{h,t}) (\phi_{h,t} - w - t\phi_h) \, dx \\
&\quad + t \int_{\mathbb{H}^N} \int_{\mathbb{H}^N} |\phi_{h,t}(x) - \phi_{h,t}(y)|^{p-2} (\phi_{h,t}(x) - \phi_{h,t}(y)) (\phi_h(x) - \phi_h(y)) \, d\mu \\
&\quad - t \int_\Omega f(x) \Phi_k'(\phi_{h,t}) \phi_h \, dx,
\end{split}
\end{equation}
where, if $J_p(t) := |t|^{p-2}t$, we define
\[
\mathcal{G}(x,y) := \frac{J_p(\phi_{h,t}(x) - \phi_{h,t}(y)) \left[ (\phi_{h,t}(x) - w(x) - t\phi_h(x)) - (\phi_{h,t}(y) - w(y) - t\phi_h(y)) \right]}{|y^{-1} \circ x|^{Q + sp}}.
\]

For future use, define also
\[
\mathcal{G}_v(x,y) := \frac{I_p(v(x) - v(y)) \left[ (\phi_{h,t}(x) - w(x) - t\phi_h(x)) - (\phi_{h,t}(y) - w(y) - t\phi_h(y)) \right]}{|y^{-1} \circ x|^{Q + sp}}.
\]

Let $S_v := \{ \phi_{h,t} = v \}$, i.e., $S_v = \{ v \leq w + t\phi_h \}$.  
Decompose: 
\[
\mathbb{H}^{N}\times \mathbb{H}^{N} = (S_v \cup S_v^c) \times (S_v \cup S_v^c).
\]
Using that $\mathcal{G} = 0$ in $S_v^c \times S_v^c$, we deduce:
\[
\begin{split}
\int_{\mathbb{H}^{N}} \int_{\mathbb{H}^{N}} \mathcal{G}(x,y) \, dx \, dy &= \int_{S_v} \int_{S_v} \mathcal{G}(x,y) \, dx \, dy + \int_{S_v^c} \int_{S_v} \mathcal{G}(x,y) \, dx \, dy + \int_{S_v} \int_{S_v^c} \mathcal{G}(x,y) \, dx \, dy \\
&\leq \int_{\mathbb{H}^{N}}\int_{\mathbb{H}^{N}} \mathcal{G}_v(x,y) \, dx \, dy.
\end{split}
\]

This follows from the monotonicity of the map $t \mapsto |t - t_0|^{p-2}(t - t_0)$.

Returning to \eqref{cicciofriccio}, we now get:
\[
\begin{split}
\mathbb{I}_1 &- \int_\Omega f(x)(\Phi_k'(\phi_{h,t}) - \Phi_k'(w))(\phi_{h,t} - w) \, dx \\
&\leq \int_{\mathbb{H}^{N}}\int_{\mathbb{H}^{N}} \mathcal{G}_v(x,y) \, dx \, dy - \int_\Omega f(x) \Phi_k'(\phi_{h,t}) (\phi_{h,t} - w - t\phi_h) \, dx \\
&\quad + t \int_{\mathbb{H}^{N}}\int_{\mathbb{H}^{N}} {|\phi_{h,t}(x) - \phi_{h,t}(y)|^{p-2} (\phi_{h,t}(x) - \phi_{h,t}(y)) (\phi_h(x) - \phi_h(y))}\,d\mu \\
&\quad - t \int_\Omega f(x) \Phi_k'(\phi_{h,t}) \phi_h \, dx.
\end{split}
\]
Using that $\phi_{h,t} - w - t\phi_h \leq 0$ and $v$ is a weak supersolution to $\mathcal{L}_{s,p}z = \Phi_k'(z)$, we deduce:
\[
\begin{split}
\mathbb{I}_1 &- \int_\Omega f(x)(\Phi_k'(\phi_{h,t}) - \Phi_k'(w))(\phi_{h,t} - w) \, dx \\
&\leq t \int_{\mathbb{H}^N} \int_{\mathbb{H}^N} {|\phi_{h,t}(x) - \phi_{h,t}(y)|^{p-2} (\phi_{h,t}(x) - \phi_{h,t}(y)) (\phi_h(x) - \phi_h(y))}\,d\mu \\
&\quad - t \int_\Omega f(x) \Phi_k'(\phi_{h,t}) \phi_h \, dx.
\end{split}
\]

Since $\phi_{h,t} - w \leq t\phi_h$, we conclude:
\[
\begin{split}
&\int_{\mathbb{H}^N} \int_{\mathbb{H}^N} {|\phi_{h,t}(x) - \phi_{h,t}(y)|^{p-2} (\phi_{h,t}(x) - \phi_{h,t}(y)) (\phi_h(x) - \phi_h(y))}\,d\mu \\
&\quad - \int_\Omega f(x) \Phi_k'(\phi_{h,t}) \phi_h\,dx \geq - \int_\Omega f(x) |\Phi_k'(\phi_{h,t}) - \Phi_k'(w)| |\phi_h| \, dx.
\end{split}
\]

Passing to the limit as $t \to 0$ and using the Lebesgue theorem, we obtain:
\[
\int_{\mathbb{H}^N} \int_{\mathbb{H}^N} {|w(x) - w(y)|^{p-2} (w(x) - w(y)) (\phi_h(x) - \phi_h(y))}\,d\mu - \int_\Omega f(x) \Phi_k'(w) \phi_h \geq 0.
\]

The claim, i.e., \eqref{ksjfskgfdfjigf}, follows by letting $h \to \infty$.
\end{proof}
Now we are in a position to prove our \emph{weak comparison principle}, namely we have the following:

\begin{Theorem}\label{comparison}
Let $f \in L^1(\Omega)\setminus\{0\}$ be nonnegative. Let $u$ be a weak subsolution to the problem \eqref{meqn} such that $u \leq 0$ on $\partial\Omega$, and let $v$ be a weak supersolution to \eqref{meqn}. Then, $u \leq v$ almost everywhere in $\Omega$.
\end{Theorem}
\begin{proof}
For $\varepsilon > 0$ and $w$ as in Lemma \ref{lemmause}, it follows that
\[
(u - w - \varepsilon)^+ \in HW^{s,p}_0(\Omega)\,.
\]
This is easily deduced by the fact that $w \in HW^{s,p}_0(\Omega)$ and $w \geq 0$ a.e. in $\Omega$, so that the support of $(u - w - \varepsilon)^+$ is contained in the support of $(u - \varepsilon)^+$.  
Therefore, by \eqref{ksjfskgfdfjigf} and standard density arguments, it follows:
\begin{equation}\label{eq111}
\begin{split}
&\int_{\mathbb{H}^N} \int_{\mathbb{H}^N} |w(x) - w(y)|^{p-2} (w(x) - w(y)) \\
&\qquad \cdot \left( T_\tau\left((u - w - \varepsilon)^+\right)(x) - T_\tau\left((u - w - \varepsilon)^+\right)(y) \right) \, d\mu \\
&\geq \int_\Omega f(x) \cdot \Phi_k'(w) T_\tau\left((u - w - \varepsilon)^+\right)\, dx
\end{split}
\end{equation}
for $T_\tau(s) := \min\{s, \tau\}$ for $s \geq 0$, and $T_\tau(-s) := -T_\tau(s)$ for $s < 0$.

Let now $\phi_n \in C_c^1(\Omega)$ such that $\phi_n \to (u - w - \varepsilon)^+$ in $HW^{s,p}_0(\Omega)$, and set
\[
\tilde\phi_{\tau,n} := T_\tau\left( \min\left\{(u - w - \varepsilon)^+, \phi_n^+ \right\} \right)\,.
\]
It follows that $\tilde\phi_{\tau,n} \in HW^{s,p}_0(\Omega) \cap L_c^\infty(\Omega)$, so that, by a density argument,
\begin{equation}\nonumber
\int_{\mathbb{H}^N} \int_{\mathbb{H}^N} |u(x) - u(y)|^{p-2} (u(x) - u(y)) (\tilde\phi_{\tau,n}(x) - \tilde\phi_{\tau,n}(y)) \, d\mu
\leq \int_\Omega {f(x)}{u^{-\delta}} \tilde\phi_{\tau,n} \, dx\,.
\end{equation}
Passing to the limit as $n \to \infty$, we easily deduce that
\begin{equation}\label{eq222}
\begin{split}
&\int_{\mathbb{H}^N} \int_{\mathbb{H}^N} |u(x) - u(y)|^{p-2} (u(x) - u(y)) \\
&\qquad \cdot \left( T_\tau\left((u - w - \varepsilon)^+(x)\right) - T_\tau\left((u - w - \varepsilon)^+(y)\right) \right) \, d\mu \\
&\leq \int_\Omega \frac{f(x)}{u^\delta} T_\tau\left((u - w - \varepsilon)^+\right)\, dx\,.
\end{split}
\end{equation}

Now set
\[
g(t) := T_\tau((t - \varepsilon)^+) = \min\left\{ \tau, \max\{t - \varepsilon, 0\} \right\}.
\]

With this notation, we have
\begin{equation}\nonumber
\begin{split}
&|u(x) - u(y)|^{p-2} (u(x) - u(y)) \left( T_\tau\left((u - w - \varepsilon)^+(x)\right) - T_\tau\left((u - w - \varepsilon)^+(y)\right) \right) \\
&= |u(x) - u(y)|^{p-2} (u(x) - u(y))(u(x) - w(x) - (u(y) - w(y))) H(x,y)
\end{split}
\end{equation}
with
\[
H(x,y) := \frac{g(u(x) - w(x)) - g(u(y) - w(y))}{(u(x) - w(x)) - (u(y) - w(y))}\,.
\]

Similarly, we obtain:
\begin{equation}\nonumber
\begin{split}
&|w(x) - w(y)|^{p-2} (w(x) - w(y)) \left( T_\tau\left((u - w - \varepsilon)^+(x)\right) - T_\tau\left((u - w - \varepsilon)^+(y)\right) \right) \\
&= |w(x) - w(y)|^{p-2} (w(x) - w(y))(u(x) - w(x) - (u(y) - w(y))) H(x,y)\,.
\end{split}
\end{equation}

Now, subtract \eqref{eq111} from \eqref{eq222}, choosing $\varepsilon > 0$ such that $\varepsilon^{-\beta} < k$, and deduce:
\begin{equation}\nonumber
\begin{split}
&c \int_{\mathbb{H}^N} \int_{\mathbb{H}^N} { (|u(x) - u(y)| + |w(x) - w(y)|)^{p-2} (u(x) - w(x) - (u(y) - w(y)))^2 }\,d\mu\\
&\leq \int_\Omega f(x) \cdot \left({u^{-\delta}} - \Phi_k'(w) \right) T_\tau\left((u - w - \varepsilon)^+\right)\, dx \\
&\leq \int_\Omega f(x) \cdot \left( \Phi_k'(u) - \Phi_k'(w) \right) T_\tau\left((u - w - \varepsilon)^+\right)\, dx \leq 0\,.
\end{split}
\end{equation}

Here we used standard elliptic estimates and the fact that $H(x,y) \geq 0$ due to $g$ being nondecreasing. Furthermore, we have
\begin{equation}\nonumber
\begin{split}
\int_{\mathbb{H}^N} \int_{\mathbb{H}^N} (|u(x) - u(y)| + |w(x) - w(y)|)^{p-2} \left( T_\tau\left((u - w - \varepsilon)^+(x)\right) - T_\tau\left((u - w - \varepsilon)^+(y)\right) \right)^2 \, d\mu \leq 0
\end{split}
\end{equation}

which proves that
\[
u \leq w + \varepsilon \leq v + \varepsilon \qquad \text{a.e. in } \Omega,
\]
and the result follows by letting $\varepsilon \to 0$.
\end{proof}

\subsection{Preliminaries for the Sobolev type inequality with extremal}
Throughout this subsection, we assume that $f \in L^m(\Omega) \setminus \{0\}$, where $m = \left( \frac{p_s^*}{1 - \delta} \right)'$, unless stated otherwise. Let $u_n \in HW_0^{s,p}(\Omega)$ be the solution to problem \eqref{approxeqn} as provided by Lemma \ref{approx}, and denote by $u_\delta$ the pointwise limit of $u_n$ in $\Omega$. By Theorem \ref{thm1}-$(a)$, we have that $u_\delta \in HW_0^{s,p}(\Omega)$ is the weak solution to problem \eqref{meqn}.

Moreover, by Remark \ref{rmkapprox}, it holds that $u_n \leq u_\delta$ in $\Omega$ for all $n \in \mathbb{N}$. 

In what follows, we establish several auxiliary results that will be useful in proving Theorem \ref{thm5}.

First, we establish the following result, which justifies the use of test functions from the space $HW_0^{s,p}(\Omega)$ in the weak formulation of equation \eqref{meqn}.

\begin{Lemma}\label{testfn}
Let $\delta > 0$, let $f \in L^1(\Omega) \setminus \{0\}$ be nonnegative, and let $u \in HW_0^{s,p}(\Omega)$ be a weak solution of problem \eqref{meqn}. Then, equation \eqref{wksoleqn} holds for every $\phi \in HW_0^{s,p}(\Omega)$.
\end{Lemma}

\begin{proof}
Assume that $u \in HW_0^{s,p}(\Omega)$ is a solution to problem \eqref{meqn}. Then, for every $\phi \in C_c^1(\Omega)$, we have
\begin{equation}\label{smthtest}
\int\limits_{\mathbb{H}^N}\int\limits_{\mathbb{H}^N}J_p(u(x)-u(y)){\big(\phi(x)-\phi(y)\big)}\,d\mu=\int_{\Omega}{f(x)}{u(x)^{-\delta}}\phi(x)\,dx.
\end{equation}
By the density of $C_c^1(\Omega)$ in $HW_0^{s,p}(\Omega)$, for every $\psi \in HW_0^{s,p}(\Omega)$, there exists a sequence of nonnegative functions $\psi_n \in C_c^1(\Omega)$ such that $\psi_n \to |\psi|$ strongly in $HW_0^{s,p}(\Omega)$ as $n \to \infty$, and $\psi_n \to |\psi|$ pointwise almost everywhere in $\Omega$. We note that
\begin{multline}
\label{unique}
\Big|\int_{\Omega}{f(x)}{u(x)^{-\delta}}\psi\,dx\Big|\leq \int_{\Omega}{f(x)}{u(x)^{-\delta}}|\psi|\,dx\leq\liminf_{n\to\infty}\int_{\Omega}{f(x)}{u(x)^{-\delta}}\psi_n\,dx
=\liminf_{n\to\infty}\langle \mathcal{L}_{s,p}u,\psi_n \rangle\\
=\int\limits_{\mathbb{H}^N}\int\limits_{\mathbb{H}^N}J_p(u(x)-u(y)){\big(\psi_n(x)-\psi_n(y)\big)}\,d\mu\\
\leq C\|u\|^{p-1}\lim_{n\to\infty}\|\psi_n\|\leq C\|u\|^{p-1}\||\psi|\|
\leq C\|u\|^{p-1}\|\psi\|,
\end{multline}
for some positive constant $C$ independent of $n$. Let $\phi \in HW_0^{s,p}(\Omega)$. Then, there exists a sequence $\phi_n \in C_c^1(\Omega)$ such that $\phi_n \to \phi$ strongly in $HW_0^{s,p}(\Omega)$. Substituting $\psi = \phi_n - \phi$ into equation \eqref{unique}, we obtain  
\begin{equation}\label{lhslim}
\lim_{n\to\infty}\Bigg|\int_{\Omega}{f(x)}{u(x)^{-\delta}}(\phi_n-\phi)\,dx\Bigg|\leq C\|u\|^{p-1}\lim_{n\to\infty}\|\phi_n-\phi\|=0.
\end{equation}
Since $\phi_n \to \phi$ strongly in $HW_0^{s,p}(\Omega)$ as $n \to \infty$, it follows that
\begin{equation}\label{rhslim}
\lim_{n \to \infty} \left\{ \int\limits_{\mathbb{H}^N} \int\limits_{\mathbb{H}^N} J_p(u(x) - u(y)) \big( (\phi_n - \phi)(x) - (\phi_n - \phi)(y) \big)\, d\mu \right\} = 0.
\end{equation}
Therefore, combining \eqref{lhslim} and \eqref{rhslim} in \eqref{smthtest}, the desired result follows.
\end{proof}

As a consequence of Lemma \ref{testfn}, we can now present a straightforward proof of the uniqueness result.

\begin{Corollary}\label{tstcor}
Let $\delta > 0$ and let $f \in L^1(\Omega) \setminus \{0\}$ be nonnegative. Then, problem \eqref{meqn} admits at most one weak solution in $HW_0^{s,p}(\Omega)$.
\end{Corollary}

\begin{proof}
Suppose, by contradiction, that $u_1, u_2 \in HW_0^{s,p}(\Omega)$ are two weak solutions of problem \eqref{meqn}. Then, by Lemma \ref{testfn}, for every $\phi \in HW_0^{s,p}(\Omega)$, we have
\begin{equation}\label{tstfnap1}
\int\limits_{\mathbb{H}^N} \int\limits_{\mathbb{H}^N} J_p(u_1(x) - u_1(y)) \big( \phi(x) - \phi(y) \big) \, d\mu = \int_{\Omega}{f(x)}{u_1(x)^{-\delta}} \phi(x) \, dx,
\end{equation}
and
\begin{equation}\label{tstfnap2}
\int\limits_{\mathbb{H}^N} \int\limits_{\mathbb{H}^N} J_p(u_2(x) - u_2(y)) \big( \phi(x) - \phi(y) \big) \, d\mu = \int_{\Omega}{f(x)}{u_2(x)^{-\delta}} \phi(x) \, dx.
\end{equation}

Now, take $\phi = (u_1 - u_2)^+ \in HW_0^{s,p}(\Omega)$ as a test function in both \eqref{tstfnap1} and \eqref{tstfnap2}, and subtract the resulting identities. This gives
\begin{equation}\label{tstfnap}
\begin{split}
& \int\limits_{\mathbb{H}^N} \int\limits_{\mathbb{H}^N} \left\{ J_p(u_1(x) - u_1(y)) - J_p(u_2(x) - u_2(y)) \right\} \\
& \qquad \times \big( (u_1 - u_2)^+(x) - (u_1 - u_2)^+(y) \big) \, d\mu \\
& = \int_{\Omega} f(x) \left\{ {u_1(x)^{-\delta}} - {u_2(x)^{-\delta}} \right\} (u_1 - u_2)^+(x) \, dx \leq 0.
\end{split}
\end{equation}
Now, arguing exactly as in the proof of monotonicity of $u_n$ in Lemma \ref{approx} above, we conclude that $u_2 \geq u_1$ in $\Omega$. Reversing the roles of $u_1$ and $u_2$, we similarly obtain $u_1 \geq u_2$ in $\Omega$. Hence, $u_1 = u_2$ in $\Omega$, which completes the proof.
\end{proof}

\begin{Remark}\label{Rmkuni}
Corollary \ref{tstcor} provides an alternative approach to establish the uniqueness of solutions in $HW_0^{s,p}(\Omega)$, as compared to Theorem \ref{thm4}.
\end{Remark}

\begin{Lemma}\label{lemma1}
Let $n \in \mathbb{N}$. Then, for every $\phi \in HW_0^{s,p}(\Omega)$, the following inequality holds:
\begin{equation}\label{prop1}
\|u_n\|^{p} \leq \|\phi\|^{p} + p \int_{\Omega} {(u_n - \phi)}{\left(u_n + \frac{1}{n}\right)^{-\delta}} f_n \, dx.
\end{equation}
Moreover, the sequence $\{\|u_n\|\}_{n\in\mathbb{N}}$ is nondecreasing, i.e.,
\begin{equation}\label{nmon}
\|u_n\| \leq \|u_{n+1}\| \quad \text{for every } n \in \mathbb{N}.
\end{equation}
\end{Lemma}

\begin{proof}
Let $h \in HW_0^{s,p}(\Omega)$. By Lemma \ref{auxresult}, there exists a unique solution $v \in HW_0^{s,p}(\Omega)$ to the problem
\[
\mathcal{L}_{s,p} v = {f_n(x)}{\Big(h^{+} + \frac{1}{n}\Big)^{-\delta}}, \quad v > 0 \text{ in } \Omega, \quad v = 0 \text{ in } \mathbb{H}^N \setminus \Omega.
\]

Furthermore, $v$ minimizes the functional $J: HW_0^{s,p}(\Omega) \to \mathbb{R}$ defined by
\[
J(\phi) := \frac{1}{p} \|\phi\|^{p} - \int_{\Omega} {f_n}{\Big(h^{+} + \frac{1}{n}\Big)^{-\delta}} \phi \, dx.
\]

Hence, for every $\phi \in HW_0^{s,p}(\Omega)$, we have $J(v) \leq J(\phi)$, which yields
\begin{equation}\label{mineqn}
\frac{1}{p} \|v\|^{p} - \int_{\Omega} \frac{f_n}{(h^{+} + \frac{1}{n})^\delta} v \, dx \leq \frac{1}{p} \|\phi\|^{p} - \int_{\Omega} {f_n}{\Big(h^{+} + \frac{1}{n}\Big)^{-\delta}} \phi \, dx.
\end{equation}

The inequality \eqref{prop1} follows by taking $v = h = u_n$ in \eqref{mineqn}. Next, choosing $\phi = u_{n+1}$ in \eqref{prop1} and using the monotonicity property $u_n \leq u_{n+1}$ from Lemma \ref{approx}, we conclude that $\|u_n\| \leq \|u_{n+1}\|$.
\end{proof}

\begin{Lemma}\label{strong}
Up to a subsequence, the sequence $\{u_n\}_{n\in\mathbb{N}}$ converges strongly to $u_\delta$ in $HW_0^{s,p}(\Omega)$.
\end{Lemma}

\begin{proof}
Observe that $u_n \leq u_\delta$. By choosing $\phi = u_\delta$ in \eqref{prop1}, we get
\[
\|u_n\| \leq \|u_\delta\|,
\]
which, together with the monotonicity property \eqref{nmon} from Lemma \ref{lemma1}, implies
\begin{equation}\label{lim1}
\lim_{n \to \infty} \|u_n\| \leq \|u_\delta\|.
\end{equation}

By Lemma \ref{apun2}-$(i)$, the sequence $\{u_n\}_{n\in\mathbb{N}}$ is uniformly bounded in $HW_0^{s,p}(\Omega)$. Therefore, up to a subsequence, $u_n \rightharpoonup u_\delta$ weakly in $HW_0^{s,p}(\Omega)$, and consequently,
\begin{equation}\label{lim2}
\|u_\delta\| \leq \lim_{n \to \infty} \|u_n\|.
\end{equation}

Combining \eqref{lim1} and \eqref{lim2} and using the uniform convexity of $HW_0^{s,p}(\Omega)$, the claim follows.
\end{proof}

\begin{Lemma}\label{minprop}
Define the functional $I_{\delta} : W_0^{1,p}(\Omega) \to \mathbb{R}$ by
\[
I_{\delta}(v) := \frac{1}{p} \|v\|^p - \frac{1}{1-\delta} \int_{\Omega} (v^{+})^{1-\delta} f \, dx.
\]
Then, $u_{\delta}$ is a minimizer of $I_{\delta}$.
\end{Lemma}

\begin{proof}
Consider the auxiliary functional $I_n : HW_0^{s,p}(\Omega) \to \mathbb{R}$ defined by
\[
I_n(v) := \frac{1}{p} \|v\|^p - \int_{\Omega} G_n(v) f_n \, dx,
\]
where
\[
G_n(t) := \frac{1}{1-\delta} \left(t^{+} + \frac{1}{n}\right)^{1-\delta} - \left(\frac{1}{n}\right)^{-\delta} t^{-}.
\]
Note that $I_n$ is coercive, bounded from below, and of class $C^1$. Therefore, $I_n$ admits a minimizer $v_n \in W_0^{1,p}(\Omega)$ such that
\[
\langle I_n'(v_n), \phi \rangle = 0 \quad \text{for all } \phi \in HW_0^{s,p}(\Omega).
\]

Since $I_n(v_n) \leq I_n(v_n^{+})$, it follows that $v_n \geq 0$ in $\Omega$. Hence, $v_n$ solves the approximated problem \eqref{approxeqn}. By the uniqueness result in Lemma \ref{approx}, we have $u_n = v_n$, implying that $u_n$ is a minimizer of $I_n$. Thus, for every $v \in HW_0^{s,p}(\Omega)$,
\begin{equation}\label{min}
I_n(u_n) \leq I_n(v^{+}).
\end{equation}

Since $u_n \leq u_\delta$, applying the Lebesgue dominated convergence theorem yields
\[
\lim_{n \to \infty} \int_{\Omega} G_n(u_n) f_n \, dx = \frac{1}{1-\delta} \int_{\Omega} u_\delta^{1-\delta} f \, dx.
\]

Moreover, by Lemma \ref{strong},
\[
\lim_{n \to \infty} \|u_n\| = \|u_\delta\|.
\]

Therefore,
\begin{equation}\label{newlim2}
\lim_{n \to \infty} I_n(u_n) = I_\delta(u_\delta).
\end{equation}

Additionally, for any $v \in W_0^{1,p}(\Omega)$,
\begin{equation}\label{lim3}
\lim_{n \to \infty} \int_{\Omega} G_n(v^{+}) f_n \, dx = \frac{1}{1-\delta} \int_{\Omega} (v^{+})^{1-\delta} f \, dx.
\end{equation}

Since $\|v^{+}\| \leq \|v\|$, using \eqref{newlim2} and \eqref{lim3} in \eqref{min}, we conclude that
\[
I_\delta(u_\delta) \leq I_\delta(v) \quad \text{for all } v \in HW_0^{s,p}(\Omega).
\]
This completes the proof.
\end{proof}

\section{Proof of the main results}
\subsection{Proof of the existence results}
\textbf{Proof of Theorem \ref{varthm1}:} We prove the theorem for $\delta_*=1$ and $\delta_*>1$ separately.
\begin{itemize}
    \item Let $\delta_*=1$. Then by Lemma \ref{apun1}-$(i)$, the sequence of solutions $\{u_n\}_{n\in \mathbb N}\subset HW^{s,p}_0(\Omega)\cap L^\infty(\Omega)$ of the problem~\eqref{approxeqn}
provided by Lemma \ref{approx} is uniformly bounded in ${HW^{s,p}_0(\Omega)}$.
Consequently, by Lemma \ref{emb}, it follows that up to a subsequence, we have $u_n \rightharpoonup u$ weakly in $HW^{s,p}_0(\Omega),$
$u_n \to u$ strongly in $L^r(\Omega)$ for $1\le r< p_s^\ast$ and $u_n \rightarrow u$ pointwise a.e.\ in $\Omega$
and, furthermore, by Lemma~\ref{approx}, we have
$$
\text{for all $K\Subset\Omega$ there exists a constant $C=C(K)>0$ such that $u(x)\geq C>0$,\quad
	for a.e. $x\in K$.}
$$	
We have
\begin{equation}\label{eq:equazPn}
\int_{\mathbb{H}^{N}}\int_{\mathbb{H}^{N}}\frac{J_p(u_n(x) - u_n(y))\, (\phi(x) - \phi(y))}{|y^{-1}\circ x|^{Q+sp}}\,dx dy
=\int_{\Omega}{f_n(x)}{\Big(u_n+\frac{1}{n}\Big)^{-\delta(x)}}\,\phi\,dx,
\end{equation}
for all $\phi \in C^1_c(\Omega)$. Since the sequence
$$
\left\{\frac{J_p(u_n(x)-u_n(y))}{|y^{-1}\circ x|^\frac{Q+sp}{p'}}\right\}_{n\in {\mathbb N}}\quad \text{is bounded in $L^{p'}({\mathbb H}^{2N})$},
$$
and by the pointwise convergence of $u_n$ to $u$
$$
\frac{J_p(u_n(x)-u_n(y))}{|y^{-1}\circ x|^\frac{Q+sp}{p'}}\to\frac{J_p(u(x)-u(y))}{|y^{-1}\circ x|^\frac{Q+sp}{p'}}
\quad \text{ pointwise a.e. in ${\mathbb H}^{2N}$},
$$
it follows by standard results that
$$
\frac{J_p(u_n(x)-u_n(y))}{|y^{-1}\circ x|^\frac{Q+sp}{p'}}\to\frac{J_p(u(x)-u(y))}{|y^{-1}\circ x|^\frac{Q+sp}{p'}}
\quad \text{weakly in $L^{p'}({\mathbb H}^{2N})$}.
$$
Then, since for $\phi \in C^1_c(\Omega)$ we have
$$
\frac{\phi(x) - \phi(y)}{|y^{-1}\circ x|^{\frac{Q+s\,p}{p}}}\in L^p({\mathbb H}^{2N}),
$$
we conclude that
\begin{equation}\label{lhs}
\begin{split}
&\lim_{n \rightarrow +\infty}\int_{\mathbb{H}^{N}}\int_{\mathbb{H}^{N}} \frac{J_p(u_n(x)-u_n(y))(\phi(x)-\phi(y))}{|y^{-1}\circ x|^\frac{Q+sp}{p'}}\,dx dy\\
&=\int_{\mathbb{H}^{N}}\int_{\mathbb{H}^{N}} \frac{J_p(u(x)-u(y))(\phi(x)-\phi(y))}{|y^{-1}\circ x|^\frac{Q+sp}{p'}}\, dx\, dy,
\end{split}
\end{equation}
for all $\phi \in C^1_c(\Omega)$.
Concerning the right-hand side of formula \eqref{eq:equazPn}, recalling Lemma \ref{approx}, for any $\phi \in C^1_c(\Omega)$
with ${\rm supp}(\phi)=K$, there exists $C=C(K)>0$ independent of $n$ such that
\begin{equation*}
\left |{f_n(x)\,\phi}{\Big(u_n+\frac{1}{n}\Big)^{-\delta(x)}} \right|\leq\big\|C^{-\delta(x)}\big\|_{L^\infty(\Omega)} |f(x)\phi(x)| \in L^1(\Omega).
\end{equation*}
By the Lebesgue's dominated convergence theorem we have
\begin{equation}\label{dct}
\lim_{n\to\infty}\int_{\Omega}{f_n(x)}{\Big(u_n+\frac{1}{n}\Big)^{-\delta(x)}} \,\phi\, dx=\int_{\Omega}{f(x)}{u^{-\delta(x)}} \,\phi\, dx.
\end{equation}
Finally, using \eqref{lhs} and \eqref{dct} in \eqref{eq:equazPn}, we conclude that
\begin{equation*}
\int_{\mathbb{H}^{N}}\int_{\mathbb{H}^{N}}\frac{J_p((u(x) - u(y))\, (\phi(x) - \phi(y))}{|y^{-1}\circ x|^{Q+s\,p}}\, dx\, dy
=\int_{\Omega}{f(x)}{u^{-\delta(x)}}\,\phi\, dx,
\end{equation*}
for all $\phi \in C^1_c(\Omega)$, namely $u\in HW_0^{s,p}(\Om)$ is a weak solution to \eqref{meqn}.

\item Let $\delta_*>1$. Then by Lemma \ref{apun1}-$(ii)$, the sequence $\{u_n\}_{n\in {\mathbb N}}$ of
	solution to \eqref{approxeqn} of Lemma \ref{approx} satisfies
	\begin{equation}\label{eq:gerry1}
	\sup_{n\in {\mathbb N}}\left[u_n^{\frac{\delta_*+p-1}{p}}\right]_{HW^{s,p}(\mathbb{H}^N)} \leq C,
	\end{equation}
    for some positive constant $C$ independent of $n$.
Since $\{u_n\}_{n\in {\mathbb N}}$ is increasing it admits pointwise limit $u$ as $n\to\infty$. In particular, by Fatou's lemma
\begin{equation}\label{eq:gerry2}
\left[u^\frac{\delta_*+p-1}{p}\right]_{HW^{s,p}(\mathbb{R}^N)}\leq \liminf_n\left[u_n^\frac{\delta_*+p-1}{p}\right]_{HW^{s,p}(\mathbb{R}^N)} \leq C.
\end{equation}
Then $u^\frac{\delta_*+p-1}{p}\in HW^{s,p}_0(\Omega)$ and, in particular $u\in L^p(\Omega)$ by Lemma \ref{emb}. Further, by Lemma \ref{mainappos}, $u^\frac{\delta_*+p-1}{p}\in HW^{s,p}_0(\Omega)$ gives that $u=0$ on $\partial\Om$ as in Definition \ref{rediri}.
Moreover, by Lemma~\ref{approx},
for all $K\Subset\Omega$ there exists $C=C(K)>0$ such that $u(x)\geq C>0$ for a.e.\ $x\in K$.
Therefore, in light of Lemma \ref{lem:dino}, we have
\begin{equation*}
\frac{|u(x) - u(y)|^{p}}{|y^{-1}\circ x|^{Q+s\,p}}\leq C^{1-\delta}\frac{\Big|u^{\frac{\delta+p-1}{p}}(x)-u^{\frac{\delta+p-1}{p}}(y)\Big|^p}{|y^{-1}\circ x|^{Q+s\,p}},
\qquad x,y\in K,\quad K\Subset{\mathbb H} ^{N}.
\end{equation*}
This yields
$$u \in HW^{s,p}_{{\rm loc}}(\Omega).$$ We have, for any $n\in {\mathbb N}$
\begin{equation}\label{eq:equazPn-2}
\int_{\mathbb{H}^{N}}\int_{\mathbb{H}^{N}}\frac{J_p(u_n(x) - u_n(y))\, (\phi(x) - \phi(y))}{|y^{-1}\circ x|^{Q+s\,p}}\, dx\, dy
=\int_{\Omega}{f_n(x)}{\Big(u_n+\frac{1}{n}\Big)^{-\delta(x)}} \,\phi\, dx,
\end{equation}
for all $\phi \in C^1_c(\Omega)$. We pass to the limit in \eqref{eq:equazPn-2}, by observing the elementary inequality $|J_p(\xi)-J_p(\xi')|\leq C(|\xi|+|\xi'|)^{p-2}|\xi-\xi'|$ for $\xi,\xi'\in{\mathbb R}$ with $|\xi|+|\xi'|>0,$ and therefore we obtain
\begin{equation}\label{eq:un-u}
\begin{split}
&\Bigg|\int_{\mathbb{H}^{N}}\int_{\mathbb{H}^{N}} \frac{J_p(u_n(x) - u_n(y))\, (\phi(x) - \phi(y))}{|y^{-1}\circ x|^{Q+s\,p}}\, dx\, dy-\int_{\mathbb{H}^{N}}\int_{\mathbb{H}^{N}} \frac{J_p(u(x) - u(y))\, (\phi(x) - \phi(y))}{|y^{-1}\circ x|^{Q+s\,p}}\, dx\, dy\Bigg|\\
&\leq \int_{\mathbb{H}^{N}}\int_{\mathbb{H}^{N}}  \frac{(|u_n(x) - u_n(y)|+|u(x) - u(y)|)^{p-2}|\bar u_n(x)-\bar u_n(y)||\phi(x) - \phi(y)|}{|y^{-1}\circ x|^{Q+s\,p}}\, dx\, dy,
\end{split}
\end{equation}
where $\bar u_n(x):= u_n(x)-u(x)$.

Let us fix $\theta>0$. We {\em claim} that there exist a compact set $\mathcal K \subset \mathbb{H}^{2N}$ such that
\begin{equation}\label{eq:un-u1}
 \int_{\mathbb{H}^{2N}\setminus \mathcal K}  \frac{(|u_n(x) - u_n(y)|+|u(x) - u(y)|)^{p-2}|\bar u_n(x)-\bar u_n(y)||\phi(x) - \phi(y)|}{|y^{-1}\circ x|^{Q+s\,p}}\, dx\, dy\leq \frac{\theta}{2},
\end{equation}
for all $n\in \mathbb N$. Let us set
\[
\mathcal S_\phi\,:={\rm supp} \,\phi \qquad \mathcal Q_\phi\,:=\,\mathbb{H}^{2N}\setminus \big(\mathcal S_\phi^c\times\mathcal S_\phi^c\big)\,.
\]
By triangular and H\"older's inequality, we get
\begin{eqnarray}\label{eq:un-u2}\\\nonumber
&&\int_{\mathbb{H}^{2N}\setminus \mathcal K}  \frac{(|u_n(x) - u_n(y)|+|u(x) - u(y)|)^{p-2}|\bar u_n(x)-\bar u_n(y)||\phi(x) - \phi(y)|}{|y^{-1}\circ x|^{Q+s\,p}}\, dx\, dy\\\nonumber
&&= \int_{\mathcal Q_\phi\setminus \mathcal K}  \frac{(|u_n(x) - u_n(y)|+|u(x) - u(y)|)^{p-1}|\phi(x) - \phi(y)|}{|y^{-1}\circ x|^{Q+s\,p}}\, dx\, dy\\\nonumber
&&\leq \left(\int_{\mathcal Q_\phi\setminus \mathcal K}  \frac{(|u_n(x) - u_n(y)|+|u(x) - u(y)|)^{p}}{|y^{-1}\circ x|^{N+s\,p}}\, dx\, dy\right)^{\frac{p-1}{p}}\left(\int_{\mathcal Q_\phi\setminus \mathcal K} \frac{|\phi(x) - \phi(y)|^{p}}{|y^{-1}\circ x|^{N+s\,p}}\, dx\, dy \right)^{\frac 1p}\\\nonumber
&&=\left(\int_{\mathcal Q_\phi\setminus \mathcal K} \frac{(|u_n(x) - u_n(y)|+|u(x) - u(y)|)^{p}}{|y^{-1}\circ x|^{Q+s\,p}}\, dx\, dy\right)^{\frac{p-1}{p}}\left(\int_{\mathbb{R}^{2N}\setminus \mathcal K} \frac{|\phi(x) - \phi(y)|^{p}}{|y^{-1}\circ x|^{Q+s\,p}}\, dx\, dy \right)^{\frac 1p}.
\end{eqnarray}
By Lemma \ref{approx}, there exists a constant $C=C({\mathcal S}_\phi)>0$ independent of $n$ such that $u(x)\geq C$ for a.e. $x
\in \mathcal S_\phi$. Moreover, by using Lemma \ref{lem:dino}  with $q=\frac{\delta_*+p-1}{p}$, we have
\begin{align}\label{eq:lagrangedino}
&\frac{(|u_n(x) - u_n(y)|+|u(x) - u(y)|)^{p}}{|y^{-1}\circ x|^{Q+s\,p}}  \nonumber \\
&\leq C(p)\frac{|u_n(x) - u_n(y)|^p+|u(x) - u(y)|^{p}}{|y^{-1}\circ x|^{Q+s\,p}}\\\nonumber
&\leq C(p)C^{1-
\delta_*}\frac{\Big|u_n^{\frac{\delta_*+p-1}{p}}(x)-u_n^{\frac{\delta_*+p-1}{p}}(y)\Big|^p+\Big|u^{\frac{\delta_*+p-1}{p}}(x)-u^{\frac{\delta_*+p-1}{p}}(y)\Big|^p}{|y^{-1}\circ x|^{Q+s\,p}},\quad\,\,\text{a.e.\ $(x,y)\in \mathcal Q_\phi$.}
\end{align}
Then from  \eqref{eq:un-u2} and \eqref{eq:lagrangedino} we infer that
\begin{eqnarray}\nonumber\\\nonumber
&&\int_{\mathbb{H}^{2N}\setminus \mathcal K}  \frac{(|u_n(x) - u_n(y)|+|u(x) - u(y)|)^{p-2}|\bar u_n(x)-\bar u_n(y)||\phi(x) - \phi(y)|}{|y^{-1}\circ x|^{N+s\,p}}\, dx\, dy\\\nonumber
&&\leq C\left(\int_{\mathbb{H}^{2N}} \frac{|u_n^{\frac{\delta+p-1}{p}}(x)-u_n^{\frac{\delta+p-1}{p}}(y)|^p}{|y^{-1}\circ x|^{N+s\,p}}\, dx\, dy\right)^{\frac{p-1}{p}}\left(\int_{\mathbb{R}^{2N}\setminus \mathcal K} \frac{(|\phi(x) - \phi(y)|)^{p}}{|y^{-1}\circ x|^{N+s\,p}}\, dx\, dy \right)^{\frac 1p}\\\nonumber
&&+ C\left(\int_{\mathbb{H}^{2N}} \frac{|u^{\frac{\delta+p-1}{p}}(x)-u^{\frac{\gamma+p-1}{p}}(y)|^p}{|y^{-1}\circ x|^{N+s\,p}}\, dx\, dy\right)^{\frac{p-1}{p}}\left(\int_{\mathbb{R}^{2N}\setminus \mathcal K} \frac{(|\phi(x) - \phi(y)|)^{p}}{|y^{-1}\circ x|^{N+s\,p}}\, dx\, dy \right)^{\frac 1p}
\\\nonumber &&\leq C \left(\int_{\mathbb{H}^{2N}\setminus \mathcal K} \frac{|\phi(x) - \phi(y)|^{p}}{|y^{-1}\circ x|^{N+s\,p}}\, dx\, dy \right)^{\frac 1p},
\end{eqnarray}
with $C=C(p,\delta,{\mathcal S_\phi})$ and where we  used \eqref{eq:gerry1} and \eqref{eq:gerry2}. Then, since $\phi \in C^1_c(\Omega)$, there exists  $\mathcal K=\mathcal K(\theta)$ such that \eqref{eq:un-u1} holds, proving the {\em claim}.

On the other hand,  consider now an arbitrary measurable subset $E\subset {\mathcal K}$. Arguing as in \eqref{eq:un-u2} and \eqref{eq:lagrangedino},
we reach the inequality
\begin{align}\nonumber
&\int_{E}  \frac{(|u_n(x) - u_n(y)|+|u(x) - u(y)|)^{p-2}|\bar u_n(x)-\bar u_n(y)||\phi(x) - \phi(y)|}{|y^{-1}\circ x|^{Q+s\,p}}\, dx\, dy\\\nonumber
& \leq C \left(\int_{E}\frac{|\phi(x) - \phi(y)|^{p}}{|y^{-1}\circ x|^{Q+s\,p}}\, dx\, dy \right)^{\frac 1p},
\end{align}
that is
\begin{equation}\nonumber
\int_{E}  \frac{(|u_n(x) - u_n(y)|+|u(x) - u(y)|)^{p-2}|\bar u_n(x)-\bar u_n(y)||\phi(x) - \phi(y)|}{|y^{-1}\circ x|^{Q+s\,p}}\, dx\, dy\rightarrow 0,
\end{equation}
uniformly on $n$, if the  Lebesgue measure of $E$ goes to zero. Moreover
\begin{equation*}
\frac{(|u_n(x) - u_n(y)|+|u(x) - u(y)|)^{p-2}|\bar u_n(x)-\bar u_n(y)||\phi(x) - \phi(y)|}{|y^{-1}\circ x|^{Q+s\,p}}\rightarrow 0\quad \text{a.e.\ in ${\mathbb H}^{2N}$}.
\end{equation*}
By Vitali's Theorem, given $\theta >0$,  there exists $\bar n >0$ such that, if $n\geq \bar n$, it follows
\begin{equation}\label{eq:un-u3}
\int_{\mathcal K}  \frac{(|u_n(x) - u_n(y)|+|u(x) - u(y)|)^{p-2}|\bar u_n(x)-\bar u_n(y)||\phi(x) - \phi(y)|}{|y^{-1}\circ x|^{Q+s\,p}}\, dx\, dy\leq \frac{\theta}{2}.
\end{equation}
From \eqref{eq:un-u}, using \eqref{eq:un-u1} and \eqref{eq:un-u3},
we are able to pass to the limit in the left-hand side of \eqref{eq:equazPn-2}, that is
\begin{eqnarray}\nonumber
\lim_{n \rightarrow +\infty}\int_{\mathbb{H}^{N}}\int_{\mathbb{H}^{N}} \frac{J_p(u_n(x) - u_n(y))\, (\phi(x) - \phi(y))}{|y^{-1}\circ x|^{Q+s\,p}}\, dx\, dy
=\int_{\mathbb{H}^{N}}\int_{\mathbb{H}^{N}} \frac{J_p(u(x) - u(y))\, (\phi(x) - \phi(y))}{|y^{-1}\circ x|^{Q+s\,p}}\, dx\, dy,
\end{eqnarray}
for all $\phi \in C^1_c(\Omega)$.
Finally, again applying Lebesgue's dominated convergence theorem to the right-hand side of \eqref{eq:equazPn-2}, and employing the above estimate in \eqref{eq:equazPn-2}, we obtain
\begin{equation*}
\int_{\mathbb{H}^{N}}\int_{\mathbb{H}^{N}} \frac{J_p(u(x) - u(y))\, (\phi(x) - \phi(y))}{|y^{-1}\circ x|^{Q+s\,p}}\, dx\, dy
=\int_{\Omega}{f(x)}{u^{-\delta(x)}}\,\phi\, dx,
\end{equation*}
for all $\phi \in C^1_c(\Omega)$, namely $u$ is a weak solution to \eqref{meqn}.
\end{itemize}

\textbf{Proof of Theorem \ref{thm1}:}
\begin{enumerate}
    \item[$(i)$] Let $0<\delta<1$. Then by Lemma \ref{apun2}-$(i)$, the sequence of solutions $\{u_n\}_{n\in \mathbb N}\subset HW^{s,p}_0(\Omega)$ of the problem~\eqref{approxeqn}
provided by Lemma \ref{approx} is uniformly bounded in ${HW^{s,p}_0(\Omega)}$. Now the result follows follows proceeding along the lines of the proof of Theorem \ref{varthm1} for the case $\delta_*=1$.

 \item[$(ii)$] Let $\delta=1$. Then by Lemma \ref{apun2}-$(ii)$, the sequence of solutions $\{u_n\}_{n\in \mathbb N}\subset HW^{s,p}_0(\Omega)$ of the problem~\eqref{approxeqn}
provided by Lemma \ref{approx} is uniformly bounded in ${HW^{s,p}_0(\Omega)}$. Now the result follows follows proceeding along the lines of the proof of Theorem \ref{varthm1} for the case $\delta_*=1$.

\item[$(iii)$] Let $\delta>1$. Then by Lemma \ref{apun2}-$(iii)$, the sequence $\Big\{u_n^\frac{\delta_*+p-1}{p}\Big\}_{n\in \mathbb N}\subset HW^{s,p}_0(\Omega)$ of the problem~\eqref{approxeqn}
provided by Lemma \ref{approx} is uniformly bounded in ${HW^{s,p}_0(\Omega)}$. Now the result follows follows proceeding along the lines of the proof of Theorem \ref{varthm1} for the case $\delta_*>1$.
\end{enumerate}

\subsection{Proof of the regularity results}
\textbf{Proof of Theorem \ref{regthm}:} 
\begin{enumerate}
\item[$(i)$] Choosing the test function $\phi=u_n^\gamma$ ($\gamma\geq \delta_*$ to be determined later) in the weak formulation of \eqref{approxeqn}, we have
\begin{align}\label{M7}
{\int_{\mathbb{H}^N}\int_{\mathbb{H}^N} \frac{J_p(u_n(x)-u_n(y))(u_n^\gamma(x)-u_n^\gamma(y))}{|y^{-1}\circ x|^{Q+sp}}\, dxdy}\leq \int_\Omega{v_n^\gamma f_n}{\Big(v_n+\frac{1}{n}\Big)^{-\delta(x)}}\, dx.  
\end{align}
By the condition $(P_{\epsilon,\delta_*})$, we observe that
\begin{align}\label{M8}
\int_\Omega {u_n^\gamma f_n}{\Big(u_n+\frac{1}{n}\Big)^{-\delta(x)}}\, dx&\leq \int_{\Omega\cap\Omega_\epsilon^c}{u_n^\gamma f}{\Big(u_n+\frac{1}{n}\Big)^{-\delta(x)}}\, dx+\int_{\{x\in\Omega_\epsilon: u_n\leq 1\}} {u_n^\gamma f}{\Big(u_n+\frac{1}{n}\Big)^{-\delta(x)}}\, dx\nonumber\\
&+\int_{\{x\in\Omega_\epsilon:u_n>1\}} {u_n^\gamma f}{\Big(u_n+\frac{1}{n}\Big)^{-\delta(x)}}\, dx\nonumber \\
&\leq \Big\|{C(\epsilon)^{-\delta(x)}}\Big\|_{L^\infty(\Omega)}\int_\Omega u_n^\gamma f\, dx+\int_{\{x\in\Omega_\epsilon: u_n\leq 1\}} u_n^{\gamma-\delta(x)} f\, dx\nonumber\\
&+\int_{\{x\in\Omega_\epsilon: u_n>1\}} u_n^\gamma f\, dx\nonumber\\
&\leq C\left (\int_\Omega u_n^\gamma f\, dx+\int_{\{x\in\Omega_\epsilon: u_n\leq 1\}} f\, dx+\int_{\{x\in\Omega_\epsilon: u_n>1\}} v_n^\gamma f\, dx\right )\nonumber\\
&\leq C\left (\int_\Omega u_n^\gamma f\, dx+||f||_{L^r(\Omega)}\right )\nonumber\\
&\leq C||f||_{L^r(\Omega)}\left [ 1+\left (\int_\Omega |u_n|^{r'\gamma}\, dx\right )^\frac{1}{r'} \right],
\end{align}
for some constant $C>0$ independent of $n$. Using Lemma \ref{BPalg} in (\ref{M7}) and employing the estimate and (\ref{M8}), we deduce
 \begin{align*}
     \Big\|u_n^\frac{\gamma+p-1}{p}\Big\|^p\leq C||f||_{L^r(\Omega)}\left [ 1+\left (\int_\Omega |u_n|^{r'\gamma}\, dx\right )^\frac{1}{r'} \right],
 \end{align*}
 for some constant $C>0$ independent of $n$. By Lemma \ref{emb}, we obtain
   \begin{align}\label{M9}
    \left (\int_\Omega u_n^\frac{p_s^*(\gamma+p-1)}{p}\, dx\right )^\frac{p}{p_s^*}\leq C||f||_{L^r(\Omega)}\left [ 1+\left (\int_\Omega |v_n|^{r'\gamma}\, dx\right )^\frac{1}{r'} \right],
\end{align}
for some constant $C>0$ independent of $n$. We choose $\gamma$ such that $\frac{p_s^*(\gamma+p-1)}{p}=r'\gamma$, i.e., $\gamma=\frac{N(p-1)(r-1)}{r(N-sp)-N(r-1)}$. Since $\frac{Q(\delta_*+p-1)}{Q(p-1)+\delta_* sp}\leq m<\frac{Q}{sp}$, we have $\gamma\ge\delta_*$ and $\frac{1}{r'}<\frac{p}{p_s^*}$. Thus inequality (\ref{M9}) gives that $\{u_n\}_{n\in\mathbb{N}}$ is uniformly bounded in $L^{m'\gamma}(\Omega)$.

\item[$(ii)$] Let $k\geq 1$ and define $A(k)=\{x\in\Om: u_n(x)\geq k\}$. Choosing $\phi_k(x)=(u_n-k)^{+}\in HW_0^{s,p}(\Om)$ as a test function in \eqref{approxeqn}, using H\"older's inequality with the exponents $p_s^{*'}, p_s^*$ and then by Young's inequality with exponents $p$ and $p'$ along with Lemma \ref{emb} we obtain
\begin{multline}
\label{unibdd}
\|\phi_k\|^p\leq\int_{\Om}\frac{f_n}{\big(u_n+\frac{1}{n}\big)^{\delta(x)}}\phi_k\,dx\leq\int_{A(k)}f(x)\phi_k\,dx\\
\leq\left(\int_{A(k)}f^{p_s^{*'}}\,dx\right)^\frac{1}{p_s^{*'}}\left(\int_{\Om}\phi_k^{p_s^*}\,dx\right)^\frac{1}{p_s^*}
\leq C\left(\int_{A(k)}f^{p_s^{*'}}\,dx\right)^\frac{1}{p_s^{*'}}\|\phi_k\|\\
\leq \epsilon\|\phi_k\|^p+C(\epsilon)\left(\int_{A(k)}f^{p_s^{*'}}\,dx\right)^\frac{p'}{p_s^{*'}}.
\end{multline}
To obtain the second inequality above, we have also used that since $u_n\geq k\geq 1$ on $\mathrm{supp}\,\phi$, therefore $\big|\frac{\phi}{(u_n+\frac{1}{n})^{\delta(x)}}\big|\leq \phi$ {on $\mathrm{supp}\,\phi$}. Here, $C$ is the Sobolev constant and  $C(\epsilon)>0$ is some constant depending on $\epsilon\in(0,1)$. Note that $q>\frac{Q}{sp}$ gives $q>p_s^{*'}$. Therefore, fixing $\epsilon\in(0,1)$ and again using H\"older's inequality with exponents $\frac{q}{p_s^{*'}}$ and $\big(\frac{q}{p_s^{*'}}\big)'$, we obtain
\begin{align*}
\|\phi_k\|^p\leq C\left(\int_{A(k)}f^{p_s^{*'}}\,dx\right)^\frac{p'}{p_s^{*'}}\leq C\left(\int_{A(k)}f^q\,dx\right)^\frac{p'}{q}|A(k)|^{\frac{p'}{p_s^{*'}}\frac{1}{\big(\frac{q}{p_s^{*'}}\big)'}}.
\end{align*}
Let $h>0$ be such that $1\leq k<h$. Then, $A(h)\subset A(k)$ and for any $x\in A(h)$, we have $u_n(x)\geq h$. So, $u_n(x)-k\geq h-k$ in $A(h)$. 
Noting these facts, for some constant $C>0$ (independent of $n$), we have
\begin{multline*}
(h-k)^p|A(h)|^\frac{p}{p_s^*}\leq\left(\int_{A(h)}(u_n-k)^{p_s^*}\,dx\right)^\frac{p}{p_s^*}\leq\left(\int_{A(k)}(u_n-k)^{p_s^*}\,dx\right)^\frac{p}{p_s^*}\\
\leq C\|\phi_k\|^p\leq C\|f\|_{L^q(\Om)}^{p'}|A(k)|^{\frac{p'}{p_s^{*'}}\frac{1}{\big(\frac{q}{p_s^{*'}}\big)'}}.
\end{multline*}
Thus, for some constant $C>0$ (independent of $n$), we have
$$
|A(h)|\leq C\frac{\|f\|_{L^q(\Om)}^\frac{p_s^*}{p-1}}{(h-k)^{p^*}}|A(k)|^{\alpha},\,\,\text{where}\,\,\alpha={\frac{p_s^{*}p'}{pp_s^{*'}}\frac{1}{\big(\frac{q}{p_s^{*'}}\big)'}}.
$$
Due to the assumption, $q>\frac{Q}{sp}$, we have $\alpha>1$. Hence, by \cite[Lemma B.$1$]{Stam}, we have
$$
\|u_n\|_{L^\infty(\Om)}\leq C,
$$
for some positive constant $C>0$, independent of $n$. Therefore, $u\in L^\infty(\Om)$.
    \end{enumerate}

\textbf{Proof of Theorem \ref{regthm1}:}
\begin{enumerate}
\item[$(i)$] We observe that 
\begin{itemize}
\item for $m = \big(\frac{p_s^{*}}{1-\delta}\big)'$ i.e., $(1-\delta)m' = p_s^{*}$, we have $\gamma = \frac{(\delta+p-1)m^{'}}{(pm^{'}-p_{s}^*)} = 1$ and
\item $m \in \big((\frac{p_s^{*}}{1-\delta})^{'},\frac{p_s^{*}}{p_s^{*}-p}\big)$ implies $\gamma = \frac{(\delta+p-1)m^{'}}{(pm^{'}-p_{s}^*)} > 1$. 
\end{itemize}
Note that $(p\gamma-p+1-\delta)m' = p_s^{*}\gamma$ and choosing $\phi = u_n^{p\gamma-p+1}\in HW_0^{s,p}(\Om)$ as a test function in \eqref{approxeqn}, and applying Lemma \ref{BPalg} we obtain
\begin{align*}
||u_n^{\gamma}||^{p}
&\leq ||f||_{L^m(\Omega)}\Big(\int_{\Omega}|u_n|^{p_s^{*}\gamma}\,dx\Big)^\frac{1}{m'}.
\end{align*}
By Lemma \ref{emb}, using the continuous embedding $HW_0^{s,p}(\Om) \hookrightarrow L^{p_s^{*}}(\Omega)$ and the fact $\frac{p}{p_s^{*}}-\frac{1}{m'} > 0$ we obtain
$
||u_n^\gamma||_{L^{p_s^{*}}(\Omega)} \leq c,
$
where $C$ is independent of $n$. Therefore, the sequence $\{u_n^\gamma\}_{n\in\mathbb{N}}$ is uniformly bounded in $L^t(\Omega)$ where $t = p_s^{*}\gamma$. Therefore the pointwise limit $u$ belong to $L^t(\Omega)$.

\item[$(ii)$] Let $k\geq 1$ and define $A(k)=\{x\in\Om: u_n(x)\geq k\}$. Choosing $\phi_k(x)=(u_n-k)^{+}\in HW_0^{s,p}(\Om)$ as a test function in \eqref{approxeqn}, and proceeding similarly as in the proof of Theorem \ref{regthm}, the result follows.
\end{enumerate}

\textbf{Proof of Theorem \ref{regthm2}:}
\begin{enumerate}
\item[$(i)$] Observe that $m \in \big(1,\frac{p_s^{*}}{p_s^{*}-p}\big)$ implies $\gamma = \frac{pm^{'}}{(pm^{'}-p_{s}^*)} > 1$. Now choosing $\phi = u_n^{p\gamma-p+1} \in HW_0^{s,p}(\Om)$ as a test function in \eqref{approxeqn} and from Lemma \ref{emb}, using the continuous embedding $HW_0^{s,p}(\Om)\hookrightarrow L^{p_s^{*}}(\Omega)$, the result follows arguing similarly as in the proof of Theorem \ref{regthm1}-$(i)$.
\item[$(ii)$] Follows similarly as in the proof of Theorem \ref{regthm}.
\end{enumerate}

\textbf{Proof of Theorem \ref{regthm3}:}
\begin{enumerate}
\item[$(i)$] Observe that $m \in \big(1,\frac{p_s^{*}}{p_s^{*}-p}\big)$ implies $\gamma = \frac{(\delta+p-1)m'}{pm'-p_s^*} > \frac{\delta+p-1}{p} > 1$, since $\delta > 1$. Now choosing $\phi = u_n^{p\gamma-p+1} \in HW_0^{s,p}(\Om)$ as a test function in \eqref{approxeqn} and from Lemma \ref{emb}, using the continuous embedding $HW_0^{s,p}(\Om)\hookrightarrow L^{p_s^{*}}(\Omega)$, the result follows arguing similarly as in the proof of Theorem \ref{regthm1}-$(i)$.
\item[$(ii)$] Follows similarly as in the proof of Theorem \ref{regthm}.
\end{enumerate}

\subsection{Proof of the uniqueness result}
\begin{proof}[Proof of Theorem \ref{main}]
If $u$ and $v$ are two weak solutions to \eqref{meqn} with zero Dirichlet boundary condition, then we have that $u \leq v$ by Theorem \ref{comparison}. In the same way, it follows that $v \leq u$.
\end{proof}

\subsection{Proof of the Sobolev type inequality with extremal}
\textbf{Proof of Theorem \ref{thm5}:}  
We recall the norm on $HW_0^{s,p}(\Omega)$ given by  
\[
\|v\| = \left( \int_{\mathbb{H}^N} \int_{\mathbb{H}^N} \frac{|v(x) - v(y)|^p}{|y^{-1} \circ x|^{Q+sp}} \, dx \, dy \right)^{\frac{1}{p}}.
\]

\begin{enumerate}
\item[(a)] To prove the theorem, it is enough to show  
\[
\Theta(\Omega) := \inf_{v \in S_\delta} \|v\|^p = \|u_\delta\|^{\frac{p(1-\delta-p)}{1-\delta}},
\]
where  
\[
S_\delta = \left\{ v \in HW_0^{s,p}(\Omega) : \int_\Omega |v|^{1-\delta} f \, dx = 1 \right\}.
\]

Define  
\[
V_\delta = \tau_\delta u_\delta, \quad \text{with} \quad \tau_\delta = \left(\int_\Omega u_\delta^{1-\delta} f \, dx\right)^{-\frac{1}{1-\delta}},
\]
so that $V_\delta \in S_\delta$.

From Lemma \ref{testfn} and by testing \eqref{wksoleqn} with $u_\delta$, we get  
\[
\|u_\delta\|^p = \int_\Omega u_\delta^{1-\delta} f \, dx.
\]

Then,  
\[
\|V_\delta\|^p = \tau_\delta^p \|u_\delta\|^p = \left( \int_\Omega u_\delta^{1-\delta} f \, dx \right)^{-\frac{p}{1-\delta}} \|u_\delta\|^p = \|u_\delta\|^{\frac{p(1-\delta-p)}{1-\delta}}.
\]

Next, for any $v \in S_\delta$, set  
\[
\lambda = \|v\|^{-\frac{p}{p+\delta-1}}.
\]

Since $u_\delta$ minimizes the functional $I_\delta$, we have  
\[
I_\delta(u_\delta) \leq I_\delta(\lambda |v|).
\]

Writing out $I_\delta$, this inequality implies  
\[
\|u_\delta\|^{\frac{p(1-\delta-p)}{1-\delta}} \leq \|v\|^p.
\]

Since $v$ was arbitrary in $S_\delta$, we conclude  
\[
\Theta(\Omega) = \inf_{v \in S_\delta} \|v\|^p = \|u_\delta\|^{\frac{p(1-\delta-p)}{1-\delta}}.
\]

\item[(b)] Assume the mixed Sobolev inequality \eqref{inequality2} holds.

If $C > \Theta(\Omega)$, then by part (a),  
\[
C \left( \int_\Omega V_\delta^{1-\delta} f \, dx \right)^{\frac{p}{1-\delta}} > \|V_\delta\|^p,
\]
which contradicts \eqref{inequality2}.

Conversely, if  
\[
C \leq \Theta(\Omega),
\]
then for any $v \in HW_0^{s,p}(\Omega) \setminus \{0\}$, define  
\[
w = \left( \int_\Omega |v|^{1-\delta} f \, dx \right)^{-\frac{1}{1-\delta}} v \in S_\delta.
\]

Thus,  
\[
C \leq \|w\|^p = \left( \int_\Omega |v|^{1-\delta} f \, dx \right)^{-\frac{p}{1-\delta}} \|v\|^p,
\]
which proves the inequality.

\item[(c)] From part (a), $\Theta(\Omega) = \|V_\delta\|^p$.

Suppose $v \in S_\delta$ also attains this minimum, i.e., $\|v\|^p = \Theta(\Omega)$.

First, $v$ cannot change sign in $\Omega$. Indeed, if $v$ changes sign, then  
\[
\||v|\|^p < \|v\|^p,
\]
but since $|v| \in S_\delta$, we would have  
\[
\Theta(\Omega) = \|v\|^p \leq \||v|\|^p,
\]
a contradiction.

Without loss of generality, assume $v \geq 0$. Define  
\[
g = \left( \int_\Omega \left( \frac{v}{2} + \frac{V_\delta}{2} \right)^{1-\delta} f \, dx \right)^{\frac{1}{1-\delta}}.
\]

By convexity,  
\[
g \geq \frac{1}{2} + \frac{1}{2} = 1.
\]

Set  
\[
h = \frac{v + V_\delta}{2g} \in S_\delta.
\]

Using the triangle inequality and strict convexity of the norm, one obtains  
\[
\Theta(\Omega) \leq \|h\|^p \leq \frac{1}{g^p} \left\| \frac{v + V_\delta}{2} \right\|^p \leq \frac{\Theta(\Omega)}{g^p} \leq \Theta(\Omega).
\]

This forces $g=1$ and equality holds in the convexity inequality, so  
\[
v = V_\delta.
\]

Therefore, the minimizer is unique up to sign.

If \eqref{sim} holds for some nonzero $w$, then  
\[
\gamma w \in S_\delta,
\]
where  
\[
\gamma = \left( \int_\Omega |w|^{1-\delta} f \, dx \right)^{-\frac{1}{1-\delta}},
\]
and hence  
\[
w = \pm \gamma^{-1} V_\delta = \pm \gamma^{-1} \tau_\delta u_\delta.
\]
\qed
\end{enumerate}

\noindent {\textsf{Prashanta Garain\\Department of Mathematical Sciences\\
Indian Institute of Science Education and Research Berhampur\\ Permanent Campus, At/Po:-Laudigam, Dist.-Ganjam\\
Odisha, India-760003
}\\ 
\textsf{e-mail}: pgarain92@gmail.com\\

\end{document}